\documentclass[leqno]{amsproc}
        \title{Non-connective $K$- and Nil-spectra of additive categories}

\author{Wolfgang L\"uck}
\email{wolfgang.lueck@him.uni-bonn.de}
       \urladdr{http://www.math.uni-muenster.de/u/lueck}
	  \address{Universit\"at Bonn\\
               Mathematisches Institut\\
               Endenicher Allee 60,
               D-53115 Bonn, Germany}
\author{Wolfgang Steimle}
\email{steimle@math.ku.dk}
       \urladdr{http://www.math.ku.dk/~gmd678}
       \address{Institut for Matematiske Fag,
			 K\o benhavns Universitet,
               Universitets\-parken~5,
               DK-2100 K\o benhavn \O, Denmark}

       \date{July 2013}
  

\usepackage[active]{srcltx}
\usepackage{pdfsync}
\usepackage{color}
\usepackage{hyperref}
\usepackage{calc}
\usepackage{enumerate,amssymb}
\usepackage{graphicx}
\usepackage[arrow,curve,matrix,tips,2cell]{xy}
  \SelectTips{eu}{10} \UseTips
  \UseAllTwocells
\usepackage{tikz}
\usetikzlibrary{decorations.pathmorphing,snakes}

\DeclareMathAlphabet{\matheurm}{U}{eur}{m}{n}

\newcommand{\addcat}{\matheurm{Add\text{-}Cat}}
\newcommand{\addcatc}{\matheurm{Add\text{-}Cat}^{\calc}}
\newcommand{\addcatt}{\matheurm{Add\text{-}Cat}_t}

\newcommand{\Spectra}{\matheurm{Spectra}}

\newcommand{\FJH}{\mathrm{FJH}}

\DeclareMathOperator{\Ch}{Ch}

\DeclareMathOperator{\cok}{cok}
\DeclareMathOperator{\colim}{colim}

\DeclareMathOperator{\ev}{ev}

\DeclareMathOperator{\func}{func}

\DeclareMathOperator{\hocolim}{hocolim}
\DeclareMathOperator{\hofib}{hofib}
\DeclareMathOperator{\hocofib}{hocofib}

\DeclareMathOperator{\id}{id}

\DeclareMathOperator{\Idem}{Idem}
\DeclareMathOperator{\map}{map}

\DeclareMathOperator{\Nil}{Nil}

\DeclareMathOperator{\ob}{ob}

\DeclareMathOperator{\pt}{pt}

\DeclareMathOperator{\PW}{PW}

\DeclareMathOperator{\funcc}{func_c}

\newcommand{\Fin}{{\mathcal{F}\text{in}}}

\newcommand{\VCyc}{{\mathcal{VC}}}

  \newcommand{\IZ}{\mathbb{Z}}

  \newcommand{\cala}{\mathcal{A}}
  \newcommand{\calb}{\mathcal{B}}
  \newcommand{\calc}{\mathcal{C}}

  \newcommand{\cali}{\mathcal{I}}
  \newcommand{\calj}{\mathcal{J}}
  \newcommand{\calk}{\mathcal{K}}

  \newcommand{\calr}{\mathcal{R}}

  \newcommand{\bfa}{{\mathbf a}}
  \newcommand{\bfainfty}{{\mathbf a}^{\infty}}
  
  \newcommand{\bfB}{{\mathbf B}}
  \newcommand{\bfb}{{\mathbf b}}
   \newcommand{\bfbinfty}{{\mathbf b}^{\infty}}

  \newcommand{\bfd}{{\mathbf d}}
  \newcommand{\bfE}{{\mathbf E}}

  \newcommand{\bfF}{{\mathbf F}}
  \newcommand{\bff}{{\mathbf f}}

  \newcommand{\bfH}{{\mathbf H}}
  \newcommand{\bfh}{{\mathbf h}}
  
  \newcommand{\bfi}{{\mathbf i}}
  \newcommand{\bfJ}{{\mathbf J}}
  \newcommand{\bfj}{{\mathbf j}}
  \newcommand{\bfK}{{\mathbf K}}
  \newcommand{\bfKch}{\mathbf K_{\operatorname{Ch}}}
  \newcommand{\bfKchinfty}{\mathbf K_{\operatorname{Ch}}^\infty}
  \newcommand{\bfKinfty}{{\mathbf K}^{\infty}}
  \newcommand{\bfKidem}{{\mathbf K}_{\Idem}}
  
  \newcommand{\bfKNilinfty}{{\mathbf K}_{\Nil}^{\infty}}

  \newcommand{\bfL}{{\mathbf L}}
  \newcommand{\bfl}{{\mathbf l}}

  \newcommand{\bfN}{{\mathbf N}}

  \newcommand{\bfs}{{\mathbf s}}
  \newcommand{\bfT}{{\mathbf T}}

  \newcommand{\bfu}{{\mathbf u}}

  \newcommand{\bfZ}{{\mathbf Z}}

\newcommand{\NK}{N\!K}

\newcommand{\bfNK}{{\mathbf N}{\mathbf K}}
\newcommand{\bfNKinfty}{{\mathbf N}{\mathbf K}^{\infty}}

\newcommand{\bfev}{{\mathbf e}{\mathbf v}}
\newcommand{\bfpr}{{\mathbf p}{\mathbf r}}
\newcommand{\bfBHS}{{\mathbf B}{\mathbf H}{\mathbf S}}

\newcommand{\inv}{^{-1}}

\newcommand{\KH}{K\!H}

\newcounter{commentcounter}

\theoremstyle{plain}
\newtheorem{theorem}{Theorem}[section]

\newtheorem{lemma}[theorem]{Lemma}
\newtheorem{corollary}[theorem]{Corollary}
\newtheorem{proposition}[theorem]{Proposition}

\newtheorem{condition}[theorem]{Condition}

\theoremstyle{definition}
\newtheorem{definition}[theorem]{Definition}

\newtheorem{example}[theorem]{Example}

\newtheorem{remark}[theorem]{Remark}

\theoremstyle{remark}

\makeatletter\let\c@equation=\c@theorem\makeatother

\newcommand{\version}[1]                       
{\begin{center} Last edited on #1\\
Last compiled on \today\\
file name: \jobname
\end{center}
}

\newcommand{\EGF}[2]{E_{#2}(#1)}

\newcommand{\xycomsquare}[8]                
{$$\xymatrix{#1 \ar[r]^-{#2} \ar[d]^{#4} &
    #3 \ar[d]^{#5}  \\
    #6\ar[r]^-{#7} & #8 }$$}

\begin{document}

\begin{abstract}
  We present an elementary construction of the non-connective algebraic $K$-theory spectrum
  associated to an additive category following the contracted functor approach due to Bass.
  It comes with a universal property that easily allows us to identify it with other
  constructions, for instance with the one of Pedersen-Weibel in terms of $\IZ^i$-graded
  objects and bounded homomorphisms.
\end{abstract}

\maketitle


\typeout{-----------------------  Introduction ------------------------}

\section*{Introduction}

In this paper we present a construction of the non-connective $K$-theory spectrum
$\bfKinfty(\cala)$ associated to an additive category $\cala$. Its $i$-th homotopy
group is the $i$-th algebraic $K$-group of $\cala$ for each $i \in \IZ$. The construction
is a spectrum version of the construction of negative $K$-groups in terms of contracted
functors due to Bass~\cite[\S7 in Chapter XII]{Bass(1968)}. 

By construction, this non-connective delooping of connective $K$-theory is the
universal one that satisfies the Bass-Heller-Swan decomposition:
roughly speaking, the passage  from the connective algebraic $K$-theory spectrum
$\bfK$ to the non-connective algebraic $K$-theory spectrum $\bfKinfty$ 
is up to weak homotopy equivalence uniquely determined by the properties
that the Bass-Heller-Swan map for $\bfKinfty$ is a weak equivalence 
and the comparison map $\bfK \to \bfKinfty$ is
bijective on homotopy groups of degree $\ge 1$. This universal property will easily
allow us to identify our model of a
non-connective algebraic $K$-theory spectrum with the construction due to
Pedersen-Weibel~\cite{Pedersen-Weibel(1985)} based on
$\IZ^i$-graded objects and bounded homomorphisms. 

We will use this construction to explain how the twisted Bass-Heller-Swan
decomposition
for connective $K$-theory of additive categories can be extended to the
non-connective
version, compare~\cite{Lueck-Steimle(2013twisted_BHS)}. We will also discuss that
the
compatibility  of the connective $K$-theory spectrum with filtered colimits
passes to the non-connective $K$-theory spectrum. Finally we deal with homotopy
$K$-theory and applications
to the $K$-theoretic Farrell-Jones Conjecture.

In the setting of functors defined on quasi-separated quasi-compact schemes, a
non-connective
delooping based on Bass's approach was carried out by Thomason-Trobaugh
\cite[section 6]{Thomason-Trobaugh(1990)}. Later Schlichting
\cite{Schlichting(2004deloop_exact), Schlichting(2006)} defined a non-connective
delooping in the wider context of exact categories and even more generally of
``Frobenius pairs''; it is the universal delooping of connective $K$-theory that
satisfies the Thomason-Trobaugh localization theorem, i.e., takes exact sequences of
triangulated categories to cofiber sequences of spectra. Most recently,
Cisinski-Tabuada \cite{Cisinski-Tabuada(2011)} and Blumberg-Gepner-Tabuada
\cite{Blumberg-Gepner-Tabuada(2013)} have used higher category theory to show that
non-connective $K$-theory, defined as a functor on dg categories and stable
$\infty$-categories, respectively, is the universal theory satisfying a list of
axioms containing the localization theorem.

While the latter characterizations are more general in that they apply to a broader
context and emphasize the role of the localization theorem, the approach presented
here has the charm of being elementary and of possessing a universal property in the
context of additive categories.

\tableofcontents

This paper has been financially supported  by the Leibniz-Award 
of the first author granted by the Deutsche Forschungsgemeinschaft. The second author was also supported by ERC Advanced Grant 288082 and the 
Danish National Research Foundation through the Centre for Symmetry and Deformation (DNRF92).
The authors want to thank the referee for his or her useful comments.


\typeout{-----------------------  Section 1: The Bass-Heller-Swan map  ------------------------}

\section{The Bass-Heller-Swan map}
\label{sec:The_Bass-Heller-Swan_map}

Let $\bfE \colon \addcat \to \Spectra$ be a covariant functor   from the category
$\addcat$ of additive categories to the category $\Spectra$ of (sequential) spectra. Our main example
will be the functor $\bfK$ which assigns to an additive category $\cala$ its (connective) algebraic
$K$-theory spectrum $\bfK(\cala)$, with the property that 
$K_i(\cala) = \pi_i(\bfK(\cala))$ for $i \ge 0$ and $\pi_i(\bfK(\cala)) = 0$ for $i \le -1$.  Let
$\cali$ be the groupoid which has two objects $0$ and $1$ and for which the set of
morphisms between any two objects consists of precisely one element.  Equip $\cala \times
\cali$ with the obvious structure of an additive category.  For an additive category
$\cala$ let $j_i \colon \cala \to \cala \times \cali$ be the functor of additive
categories which sends a morphism $f \colon A \to B$ in $\cala$ to the morphism 
$f \times \id_i \colon A \times i \to B \times i$ for $i = 0,1$.  

\begin{condition} \label{con:condition_bfu}
For every additive  category $\cala$, we require the existence of a map
\[
\bfu \colon \bfE(A) \wedge [0,1]_+ \to \bfE(\cala \times \cali)
\]
such that $\bfu$ is natural in $A$ and, for $i = 0,1$ and 
$k_i \colon \bfE(\cala) \to \bfE(\cala) \times [0,1]_+$ 
the obvious inclusion coming from the inclusion $\{i\} \to [0,1]$, the
composite $\bfu\circ k_i$ agrees with $\bfE(j_i)$. 
\end{condition}

The connective algebraic $K$-theory spectrum functor $\bfK$ satisfies 
Condition~\ref{con:condition_bfu}, cf.~\cite[Proposition~1.3.1 on page~330]{Waldhausen(1985)}.

\begin{lemma} Suppose that $\bfE \colon \addcat \to \Spectra$ satisfies Condition~\ref{con:condition_bfu}.
  \label{lem:nat_equiv_homotopy}
  \begin{enumerate}

  \item \label{lem:nat_equiv_homotopy:homotopy} 
    Let $F_i \colon \cala \to \calb$, $i=0,1$ be two
    functors of additive categories and let $T \colon F_0 \xrightarrow{\simeq} F_1$ be a
    natural isomorphism. Then we obtain a homotopy
    \[
    \bfh \colon \bfE(\cala) \wedge [0,1] \to \bfE(\calb)
    \]
    with $\bfh \circ k_i = \bfE(F_i)$ for $i =0,1$. This construction is natural in $F_0$,
    $F_1$ and $T$;

  \item \label{lem:nat_equiv_homotopy_equivalence} An equivalence $F \colon \cala \to \calb$ 
  of additive categories induces a homotopy equivalence $\bfE(F) \colon \bfE(\cala) \to \bfE(\calb)$.

  \end{enumerate}
\end{lemma}

\begin{proof}~\ref{lem:nat_equiv_homotopy:homotopy}
The triple $(F_0,F_1,T)$ induces an additive functor $H\colon \cala \times \cali \to \calb$ with
$H \circ j_i = F_i$ for $i = 0,1$. Define $\bfh$ to be the composite $\bfE(H) \circ \bfu$.
\\[1mm]~\ref{lem:nat_equiv_homotopy_equivalence}
Let $F' \colon \calb \to \cala$ be a functor such that $F' \circ F$ is naturally isomorphic to $\id_{\cala}$ and
$F \circ F'$ is naturally isomorphic to $\id_{\calb}$. Assertion~\ref{lem:nat_equiv_homotopy:homotopy}
implies $\bfE(F') \circ \bfE(F) \simeq \id_{\bfE(\cala)}$ and $\bfE(F) \circ \bfE(F') \simeq \id_{\bfE(\calb)}$.
\end{proof}

Let $\cala$ be an additive category.  Define the \emph{associated Laurent category}
$\cala[t,t^{-1}]$ as follows. It has the same objects as $\cala$.  Given two objects $A$
and $B$, a morphism $f \colon A \to B$ in $\cala[t,t^{-1}]$ is a formal sum 
$f = \sum_{k  \in \IZ} f_k \cdot t^k$, where $f_k \colon A \to B$ is a morphism in $\cala$ and only
finitely many of the morphisms $f_k$ are non-trivial. If $f = \sum_{i  \in \IZ} f_i  \cdot t^i\colon A \to B$
and $g = \sum_{j \in \IZ} g_j \cdot t^j \colon B \to C$ are morphisms in $\cala[t,t^{-1}]$, 
we define the composite $g \circ f \colon A \to C$ by
\[
g \circ f := \sum_{k \in \IZ} \;\biggl( \sum_{\substack{i,j \in \IZ,\\i+j = k}} g_j \circ   f_i\biggr) \cdot t^k.
\]
The direct sum and the structure of an abelian group on the set of morphism from $A$ to
$B$ in $\cala[t,t^{-1}]$ is defined in the obvious way using the
corresponding structures of $\cala$.

Let $\cala[t]$ and $\cala[t^{-1}]$ respectively be the additive subcategory of
$\cala[t,t^{-1}]$ whose set of objects is the set of objects in $\cala$ and whose morphism
from $A$ to $B$ are given by finite Laurent series $\sum_{k \in \IZ} f_k \cdot t^k$
with $f_k = 0$ for $k < 0$ and $k > 0$ respectively.

If $\cala$ is the additive category of finitely generated free $R$-modules, then
$\cala[t]$ and $\cala[t,t^{-1}]$ respectively are equivalent to the additive category of
finitely generated free modules over $R[t]$ and $R[t,t^{-1}]$ respectively.

Define functors
\[
z[t,t^{-1}], z[t], z[t^{-1}]  \colon \addcat \to \addcat
\]
by sending an object $\cala$ to the object $\cala[t,t^{-1}]$, $\cala[t]$ and $\cala[t^{-1}]$ respectively.  
Their definition on morphisms in $\addcat$ is obvious.
Next we define natural transformations of functors $\addcat \to \addcat$
\begin{equation}\label{eq:square_of_functors}\xymatrix{
  & z[t] \ar[rd]^{j_+} \ar@/_2ex/[ld]_{\ev_0^+}\\
\id_{\addcat}  \ar[ru]_{i_+} \ar[rd]^{i_-} \ar[rr]^{i_0}  && z[t,t\inv]\\
  & z[t\inv] \ar[ru]_{j_-} \ar@/^2ex/[lu]^{\ev_0^-}
}\end{equation}

We have to specify for every object $\cala$ in $\addcat$  their values on $\cala$.
The functors $i_0(\cala)$, $i_+(\cala)$ and $i_-(\cala)$ send a morphism $f \colon A \to B$ in
$\cala$ to the morphism $f \cdot t^0 \colon A \to B$.  The functors
$j_{\pm}(\cala)$ are the obvious inclusions.
The functor $\ev_0^{\pm}(\cala) \colon \cala_{\phi}[t^{\pm}] \to \cala$ 
sends a morphism $\sum_{k \ge 0} f_k \cdot t^k$ in $\cala[t]$
or $\sum_{k \le 0} f_k \cdot t^k$ in $\cala[t^{-1}]$
respectively to $f_0$.  Notice that $\ev_0^+\circ i_+= \ev_0^-\circ i_-= \id_{\cala}$ 
and $i_0 = j_+ \circ i_+ = j_- \circ i_-$ holds.

Given a functor $\bfE\colon\addcat\to\Spectra$, we now define a number of new functors of the same type. Put
\begin{eqnarray*}
\bfZ_{\pm}\bfE  & := & \bfE \circ z[t^{\pm1}];
\\
\bfZ\bfE &  := & \bfE \circ z[t,t^{-1}].
\end{eqnarray*} 
The square \eqref{eq:square_of_functors} induces a square of natural transformations
\[\xymatrix{
  & \bfZ_+\bfE \ar[rd]^{\bfj_+} \ar@/_2ex/[ld]_{\bfev_0^+}\\
\bfE  \ar[ru]_{\bfi_+} \ar[rd]^{\bfi_-} \ar[rr]^{\bfi_0}  && \bfZ_\pm\bfE\\
  & \bfZ_-\bfE \ar[ru]_{\bfj_-} \ar@/^2ex/[lu]^{\bfev_0^-}
}\]


Next we define a natural transformation 
\[
\bfa \colon \bfE\wedge (S^1)_+  \to   \bfZ\bfE.
\]
Let $T \colon i_0 \to i_0$ be the natural transformation of functors $\cala \to
\cala[t,t^{-1}]$ of additive categories whose value at an object $A$ is given by the
isomorphism $\id_A \cdot t \colon A \to A$ in $\cala[t,t^{-1}]$.  Because of
Lemma~\ref{lem:nat_equiv_homotopy}~\ref{lem:nat_equiv_homotopy:homotopy} it induces a
homotopy $\bfh \colon \bfE(\cala) \wedge [0,1]_+ \to \bfE(\cala[t,t^{-1}])$ such that
$\bfh_0 = \bfh_1 = \bfE(i_0)$ holds, where $\bfh_i := \bfh \circ k_i$ for the obvious
inclusion $k_i \colon \bfE(\cala) \to \bfE(\cala) \times [0,1]_+$ for $i = 0,1$. 
Since we have the obvious pushout
\[
\xymatrix{\bfE(\cala)\wedge \{0,1\}_+ \ar[r] \ar[d]
& \bfE(\cala) \wedge [0,1]_+ \ar[d]
\\
\bfE \ar[r]
&
\bfE(\cala) \wedge (S^1)_+}
\]
we obtain  a map $\bfa(\cala) \colon \bfE(\cala) \wedge (S^1)_+ \to  \bfE(\cala[t,t^{-1}])$.
This defines a transformation
\[
\bfa \colon \bfE\wedge (S^1)_+  \to   \bfZ\bfE.
\]

In order to guarantee the existence of $\bfa$, we have imposed the Condition~\ref{con:condition_bfu}
which is stronger than just demanding that $\bfE$ sends equivalences of additive categories
to (weak) homotopy equivalences of spectra.

Define $\bfN_{\pm}\bfE$ to be the homotopy fiber of the map of
spectra $\bfev_0^{\pm} \colon \bfZ^{\pm}\bfE \to \bfE$.
Let 
\[
\bfb_{\pm }  \colon \bfN_{\pm}\bfE\to \bfZ_{\pm}E \xrightarrow{\bfj_\pm} \bfZ\bfE
\]
be the composite map, where the first map is the canonical one.

Define
\begin{eqnarray*}
\bfB\bfE & = & \bigl(\bfE \wedge (S^1)_+\bigr) \vee \bfN_+\bfE \vee \bfN_-\bfE;
\\
\bfB_r\bfE & = & \bfE \vee \bfN_+\bfE \vee \bfN_-\bfE
\end{eqnarray*}
and let
\begin{eqnarray*}
\bfBHS  :=  \bfa \vee \bfb_+ \vee \bfb_- & \colon & \bfB\bfE \to \bfZ_\pm\bfE;
\\
\bfBHS_r  :=  \bfi_0 \vee \bfb_+ \vee \bfb_- & \colon & \bfB_r\bfE \to \bfZ_\pm\bfE.
\end{eqnarray*}
We sometimes call $\bfBHS$ the \emph{Bass-Heller-Swan map} and
$\bfBHS_r$ the \emph{restricted Bass-Heller-Swan map}.

We have the following commutative diagram
\begin{eqnarray}
&
\xymatrix{\bfE \ar[d] \ar[r]^-{\bfl} 
&\bfE \wedge (S^1)_+ \ar[d]^{\bfa}
\\
\bfB_r\bfE \ar[r]^-{\bfBHS_r} & 
\bfZ\bfE
}
\label{homotopy_pushout_for_bfl}
\end{eqnarray}
where the left vertical arrow is the canonical inclusion,
and $\bfl$ is the obvious inclusion coming from the 
identification $\bfE = \bfE \wedge \pt_+$ and the inclusion $\pt_+ \to (S^1)_+$.
It induces a map $\bfs'' \colon \hocofib(\bfl) \to \hocofib(\bfBHS_r)$.
Let $\bfpr \colon \hocofib(\bfl) \to \Sigma \bfE  := \bfE \wedge S^1$ be the obvious projection
which is a homotopy equivalence. Define $\bfL\bfE$ to be the homotopy pushout
\begin{eqnarray}
\xymatrix{
\hocofib(\bfl) \ar[r]^-{\bfpr}_-{\simeq} \ar[d]_{\bfs''}
& 
\Sigma \bfE \ar[d]^{\bfs'}
\\
\hocofib(\bfBHS_r) \ar[r]_-{\overline{\bfpr}}^-{\simeq}
&
\bfL\bfE
}
\label{diagram_for_bfL}
\end{eqnarray}
By construction we obtain a homotopy cofiber sequence
\[
\bfBHS_r\bfE \xrightarrow{\bfBHS_r} \bfZ\bfE \to \bfL\bfE.
\]
Denote by 
\[
\bfs \colon \bfE \to \Omega \bfL\bfE
\] 
the adjoint of $\bfs'$.

Summarizing, we have now defined functors
\[
\bfB\bfE, \bfB_r\bfE, \bfL\bfE, \bfN_{\pm}\bfE, \bfZ_{\pm}\bfE,\bfZ\bfE \colon \addcat \to \Spectra
\]
and natural transformations
\begin{eqnarray*}
\bfi_0 \colon \bfE & \to & \bfZ\bfE;
\\
\bfi_{\pm} \colon \bfE & \to & \bfZ_{\pm}\bfE;
\\
\bfj_{\pm} \colon \bfZ_{\pm} \bfE & \to & \bfZ\bfE;
\\
\bfev_0^{\pm} \colon \bfZ_{\pm} \bfE & \to & \bfE;
\\
\bfa \colon \bfE \wedge (S^1)_+ & \to & \bfZ \bfE;
\\
\bfb_{\pm} \colon \bfN_{\pm}\bfE & \to & \bfZ\bfE;
\\
\bfBHS \colon \bfB\bfE & \to & \bfZ\bfE;
\\
\bfBHS_r \colon \bfB_r\bfE & \to & \bfZ\bfE;
\\
\bfs \colon \bfE & \to & \Omega \bfL\bfE.
\end{eqnarray*}

\begin{definition}[Compatible transformations]\label{def:funcc}
  Let $E,F\colon \addcat\to\Spectra$ be two functors satisfying 
 Condition~\ref{con:condition_bfu}. A natural transformation $\phi\colon \bfE\to \bfF$ is called
  \emph{compatible} if the obvious diagram
  \[\xymatrix{
    \bfE(\cala)\wedge [0,1]_+ \ar[rr]^\bfu \ar[d]^\phi && \bfE(\cala\times\cali) \ar[d]^\phi\\
    \bfF(\cala)\wedge [0,1]_+ \ar[rr]^\bfu && \bfF(\cala\times\cali) }\]
  is commutative. The category
  \[\funcc(\addcat, \Spectra)\]
  is the category of functors satisfying Condition~\ref{con:condition_bfu} whose morphisms
  are compatible natural transformations.
\end{definition}

We leave the proof of the following lemma to the reader:

\begin{lemma}\label{lem:inheritance_of_condition_under_general_constructions}
  \begin{enumerate}
  \item If $\bfE\colon \addcat\to\Spectra$ satisfies Condition~\ref{con:condition_bfu}
    then so do $\bfE\wedge X$ and $\map(X, \bfE)$ for any space $X$.
  \item If $\bfE$ and $\bfE'$ satisfy Condition~\ref{con:condition_bfu} then so does
    $\bfE\vee\bfE'$.
  \item If
    \[\bfE_1 \rightarrow \bfE_0 \leftarrow \bfE_2\]
    is a diagram in $\funcc(\addcat, \Spectra)$ , then its homotopy pullback satisfies
    Condition~\ref{con:condition_bfu}.
  \item If $\calj$ is a small category and $F\colon \calj\to \funcc(\addcat,\Spectra)$ is
    a functor, then $\hocolim F$ satisfies Condition~\ref{con:condition_bfu}.
  \end{enumerate}
\end{lemma}

\begin{remark}
    The category $\Spectra$ is the ``naive'' one with strict morphisms of spectra as described for instance
   in~\cite{Davis-Lueck(1998)}.  Our model for $\Omega\bfE$ is the spectrum $\map(S^1, \bfE)$ defined levelwise,
    and analogously for the homotopy pushout, homotopy pullback, homotopy fiber,
   and more general for homotopy colimits and homotopy limits
   over arbitrary index categories.  For more details see for 
   instance~\cite{Davis-Lueck(1998),Hollender-Vogt(1992)}.
  \end{remark}

As an application of Lemma~\ref{lem:inheritance_of_condition_under_general_constructions}, we deduce:

\begin{lemma} \label{lem:cond_inherited}
If $\bfE \colon \addcat \to \Spectra$ satisfies Condition~\ref{con:condition_bfu}, then so do
$\bfB\bfE$, $\bfB_r\bfE$, $\bfL\bfE$, $\bfN_{\pm}\bfE$, $\bfZ_{\pm}\bfE$, and $\bfZ\bfE$.
\end{lemma}

We will apply this as well as the following result without further remarks.

\begin{lemma} \label{lem:f_equiv_implies_Bf,Zf,Nf_equiv}
Let $\bff \colon \bfE \to \bfF$ be a transformation of functors $\addcat \to \Spectra$.
Suppose that it is a weak equivalence, i.e., $\bff(\cala)$ is a weak equivalence for any object $\cala$
in $\addcat$. Then the same is true for the transformations 
$\bfB\bff$, $\bfB_r\bff$, $\bfL\bff$, $\bfN_{\pm}\bff$, $\bfZ_{\pm}\bff$, and $\bfZ\bff$.
\end{lemma}


\typeout{-----------------------  Section 2: Contracted functors ------------------------}

\section{Contracted functors}
\label{sec:Contracted_functors}

Let $\bfE \colon \addcat \to \Spectra$ be a covariant functor satisfying Condition~\ref{con:condition_bfu}.

\begin{definition}[$c$-contracted]
\label{def:c-contracted_functor}
For $c \in \IZ$, we call $\bfE$ \emph{$c$-contracted} if is satisfies the following two
conditions:

\begin{enumerate}

\item For every $i \in \IZ$ the natural transformation 
  $\pi_i(\bfBHS_r) \colon  \pi_i(\bfB_r\bfE) \to\pi_i(\bfZ\bfE)$ is split injective, i.e., there exists 
a natural   transformation of functors from $\addcat$ to the category of abelian groups
\[
\rho_i \colon \pi_i(\bfZ\bfE) \to \pi_i(\bfB_r\bfE)
\]
such that the composite $\pi_i(\bfB_r\bfE) \xrightarrow{\pi_i(\bfBHS_r)} \pi_i(\bfZ\bfE)
\xrightarrow{\rho_i} \pi_i(\bfB_r\bfE)$ is the identity; 

\item For $i \in \IZ, i \ge -c+1$ the transformation
\[
\pi_i(\bfBHS) \colon \pi_i(\bfB\bfE)  \to \pi_i(\bfZ\bfE)
\]
is an isomorphism, i.e., its evaluation at any additive category $\cala$ is bijective.
\end{enumerate}

We  call $\bfE$ \emph{$\infty$-contracted} if 
$\bfBHS \colon \bfB\bfE  \to \bfZ\bfE$ is a weak homotopy equivalence.
\end{definition}

\begin{lemma} \label{lem:c_contracted_and_vee} Let $\bfE,\bfE' \colon \addcat \to  \Spectra$ 
  be covariant functors satisfying Condition~\ref{con:condition_bfu}.  
  \begin{enumerate} 
  
  \item \label{lem:c_contracted_and_vee:pi_i_iso}
  Suppose that  $\bfE$ and $\bfE'$ satisfy Condition~\ref{con:condition_bfu}.  Consider $i \in \IZ$.
  Then both $\pi_i(\bfBHS(\bfE))$ and $\pi_i(\bfBHS(\bfE)')$ are isomorphisms if and only if
  $\pi_i(\bfBHS(\bfE \vee \bfE'))$ is an isomorphism;

  \item \label{lem:c_contracted_and_vee:contracted}
  Suppose that  $\bfE$ and $\bfE'$ satisfy Condition~\ref{con:condition_bfu}. Consider $c \in \IZ$. Then
   $\bfE \vee \bfE'$ is $c$-contracted if and
  only if both $\bfE$ and $\bfE'$ are $c$-contracted.
\end{enumerate}
\end{lemma}
\begin{proof}
The transformation
\[
\bfb_{\pm} \vee \bfi_0 \colon \bfN_{\pm}\bfE \vee \bfE \to \bfZ_{\pm}\bfE
\]
is a weak equivalence, i.e., its evaluation at any additive category $\cala$ is a  weak equivalence of spectra,
 since $\bfev_0^{\pm}  \circ \bfi_0 = \id_{\bfE}$ holds. Note that 
\[
\bfZ_{\pm}\bfE\vee \bfZ_{\pm}\bfE' = \bfZ_{\pm}(\bfE \vee \bfE')
\]
so
\[
\bfN_{\pm}\bfE \vee \bfN_{\pm}\bfE' = \bfN_{\pm}(\bfE \vee \bfE').
\]
The obvious map
\[
\bigl(\bfE \wedge (S^1)_+\bigr)  \vee \bigl(\bfE' \wedge (S^1)_+\bigr) \to (\bfE \vee \bfE')\wedge (S^1)_+
\]
is an isomorphism. Hence the following two obvious transformations are weak equivalences
\begin{eqnarray*}
\bfB_r\bfE \vee \bfB_r\bfE'  & \to & \bfB_r(\bfE \vee \bfE');
\\
\bfB\bfE \vee \bfB\bfE'  & \to & \bfB(\bfE \vee \bfE').
\end{eqnarray*}
Now the claim follows from the following two commutative diagrams
\[
\xymatrix{\pi_i(\bfB\bfE) \oplus \pi_i(\bfB\bfE')
\ar[r]^-{\cong}  \ar[d]_{\pi_i(\bfBHS(\bfE)) \oplus \pi_i(\bfBHS(\bfE'))}
&
\pi_i(\bfB(\bfE \vee \bfE'))
\ar[d]^{\pi_i(\bfBHS(\bfE \vee \bfE'))}
\\
\pi_i(\bfZ\bfE) \oplus \pi_i(\bfZ\bfE')
\ar[r]^-{\cong}  
&
\pi_i(\bfZ(\bfE \vee \bfE'))}
\]
and
\belowdisplayskip=-12pt
\[
\xymatrix{\pi_i(\bfB_r\bfE) \oplus \pi_i(\bfB_r\bfE')
\ar[r]^-{\cong}  \ar[d]_{\pi_i(\bfBHS_r(\bfE)) \oplus \pi_i(\bfBHS_r(\bfE'))}
&
\pi_i(\bfB_r(\bfE \vee \bfE'))
\ar[d]^{\pi_i(\bfBHS_r(\bfE \vee \bfE'))}
\\
\pi_i(\bfZ\bfE) \oplus \pi_i(\bfZ\bfE')
\ar[r]^-{\cong}  
&
\pi_i(\bfZ(\bfE \vee \bfE'))}
\]
\end{proof}

Define
\begin{eqnarray}
&
\bfK \colon \addcat\to \Spectra
&
\label{bfK}
\end{eqnarray}
to be the connective $K$-theory spectrum functor in the sense of Quillen~\cite[page~95]{Quillen(1973)}
by regarding  $\cala$ as an exact category or in the sense of Waldhausen~\cite[page~330]{Waldhausen(1985)}
by regarding  $\cala$ as a Waldhausen category. (These approaches are equivalent,
see~\cite[Section~1.9]{Waldhausen(1985)}).

\begin{theorem}[Bass-Heller-Swan Theorem for $\bfK$]
\label{the:Bass-Heller-Swan_Theorem_for_bfK}
The functor $\bfK$ is $1$-contract\-ed in the sense of Definition~\ref{def:c-contracted_functor}.
\end{theorem}

\begin{proof}
  The proof that the Bass-Heller-Swan map induces bijections on $\pi_i$ for $i \ge 1$ can
  be found in~\cite[ Theorem~0.4~(i)]{Lueck-Steimle(2013twisted_BHS)} provided that $\cala$ is idempotent
  complete. Denote by $\eta \colon \cala \to \Idem(\cala)$ the inclusion of
  $\cala$ into its idempotent completion. By cofinality~\cite[Theorem~A.9.1]{Thomason-Trobaugh(1990)} the maps $\bfZ\bfK(\eta)$ and
  $\bfB_r\bfK(\eta)$ induce isomorphisms on $\pi_1$ for $i\ge 1$; the map
  $\bfB\bfK(\eta)$ induces isomorphisms at least for $i\geq 2$. The commutativity of the
  diagram
\[\xymatrix{
\bfB \bfK(\cala) \ar[rr]^{\bfBHS} \ar[d]^{\bfB \bfK(\eta)}
 && \bfZ \bfK(\cala) \ar[d]^{\bfZ\bfK(\eta)}
\\
\bfB \bfK(\Idem(\cala)) \ar[rr]^{\bfBHS} 
 && \bfZ \bfK(\Idem(\cala)) 
}\]
shows that the Bass-Heller-Swan map for $\cala$ induces isomorphisms of $\pi_i$ for $i\geq
2$ and that the restricted Bass-Heller-Swan map for $\cala$ is split injective on $\pi_i$
for $i\geq 1$.

Since all spectra are connective, it remains to show that the restricted Bass-Heller-Swan
map for $\cala$ is split injective on $\pi_0$. Notice that
\[\pi_0 (\bfK(\cala))\to \pi_0( \bfK(\cala[t]))\]
is surjective as both categories $\cala$ and $\cala[t]$ have the same objects. It follows that
$\pi_0 (\bfN\bfK(\cala))=0$ and we need to show that the  map induced by the inclusion
\[\pi_0 (\bfK(\cala))\to \pi_0 (\bfK(\cala[t, t\inv]))\]
is split mono. Such a split is given by evaluation at $t=1$.
\end{proof}

Denote by $\Idem\colon \addcat\to\addcat$ the idempotent completion functor, and let 
\[\bfKidem:=\bfK\circ\Idem \colon \addcat\to \Spectra.\]

\begin{example}[Algebraic $K$-theory of a ring $R$]
  \label{exa:ring}
  Given a ring $R$, then the idempotent completion $\Idem(\calr)$ of the additive category
  $\calr$ of finitely generated free $R$-modules is equivalent to the additive category of
  finitely generated projective $R$-modules. Moreover, the map $\IZ \to \pi_0 \bfK(\calr)$
  sending $n$ to $[R^n]$ is surjective (even bijective if $R^n \cong R^m$ implies $m =
  n$), whereas $\pi_0 \bfKidem(\calr)$ is the projective class group of $R$.
\end{example}

For an additive category we define
its algebraic $K$-group
\begin{eqnarray}
K_i(\cala) &  := & \pi_i(\bfKidem(\cala)) \quad \text{for}\; i \ge 0.
\label{K_i(cala)_i_ge_0}
\end{eqnarray}
We already showed that by cofinality, the  map  induced by the inclusion
\[\pi_i \bfK(\cala) \to \pi_i \bfKidem(\cala) = K_i(\cala)\]
is an isomorphism for $i\geq 1$.

\begin{theorem}[Bass-Heller-Swan Theorem for connective algebraic $K$-theory]
\label{the:Bass-Heller-Swan_Theorem_for_bfKIdem}
The functor $\bfKidem$ is $0$-contracted in the sense of Definition~\ref{def:c-contracted_functor}.
\end{theorem}

\begin{proof}
In view of the proof of Theorem~\ref{the:Bass-Heller-Swan_Theorem_for_bfK}, 
the Bass-Heller-Swan map is bijective on $\pi_i$ for $i\geq 1$. It remains to show split injectivity on $\pi_0$. 

We will abbreviate $\calb= \cala[s,s^{-1}]$. Notice for the sequel
that 
\[\calb[t^{\pm1}] = \bigl(\cala[t^{\pm1}]\bigr)[s,s^{-1}], \quad \calb[t,t\inv] = \bigl(\cala[t,t\inv]\bigr)[s,s^{-1}].\]
Put 
\begin{eqnarray*}
\NK_i(\cala[t^{\pm1}]) 
& = & 
\pi_i(\bfN_{\pm1}\bfKidem(\cala)) = \ker\bigl(K_i(\cala[t^{\pm1}] \to K_i(\cala)\bigr);
\\
\NK_i(\calb[t^{\pm1}]) 
& = & 
\pi_i(\bfN_{\pm1}\bfKidem(\calb)) = \ker\bigl(K_i(\calb[t^{\pm1}] \to K_i(\calb)\bigr).
\end{eqnarray*}
Because of Lemma~\ref{lem:c_contracted_and_vee} also the Bass-Heller-Swan map for $\bfN_{\pm1}\bfKidem$
induces isomorphisms on $\pi_1$.

In particular we get split injections
\begin{eqnarray*}
\alpha \colon K_0(\cala) & \to & K_1(\calb);
\\
\beta_\pm \colon \NK_0(\cala[t^{\pm1}]) & \to & \NK_1(\calb[t^{\pm1}]);
\\
j \colon K_0(\cala[t,t^{-1}]) & \to & K_1(\calb[t,t^{-1}]).
\end{eqnarray*}
We  obtain the following commutative diagram
\[
\xymatrix@!C= 17em{
K_0(\cala) \oplus \NK_0(\cala[t]) \oplus \NK_0(\cala[t^{-1}]) \ar[r]^-{\pi_0(\bfBHS_r(\cala))} 
\ar[d]_{\alpha\oplus\beta_+\oplus \beta_-}
&
K_0(\cala[t,t^{-1}]) \ar[d]^{j}
\\
K_1(\calb) \oplus \NK_1(\calb[t]) \oplus \NK_1(\calb[t^{-1}])
\ar[r]^-{\pi_1(\bfBHS_r(\calb))}_-{\cong} 
&
K_1(\calb[t,t^{-1}])
}\]
which is compatible with the splitting. So $\pi_0(\bfBHS_r(\cala))$ is a split mono, 
being a retract of the split mono $\pi_1(\bfBHS_r(\calb))$. 
\end{proof}

\begin{lemma}\label{E-n-contracted_implies_LE_is_n-contracted}
  If $\bfE \colon \addcat \to \Spectra$ is $c$-contracted, then $\Omega\bfL\bfE \colon \addcat \to \Spectra$ is
 $(c+1)$-contracted and the map $\pi_i(\bfs) \colon \pi_i(\bfE) \to \pi_i(\Omega(\bfL\bfE))$ is bijective for $i \ge - c$.
\end{lemma}
\begin{proof}
Obviously it suffices to show that $\bfL\bfE$ is $c$-contracted and that the map 
$\bfs' \colon \Sigma \bfE \to \bfL\bfE$, which is the adjoint of $\bfs$,
induces an isomorphism on $\pi_i$ for $i \ge -c+1$.

Since $\bfE$ is $c$-contracted, $\bfZ\bfE$, $\bfZ_+\bfE$  and $\bfZ_-\bfE$  are $c$-contracted.
We have the obvious cofibration sequence
$ \bfE \to \bfE \wedge (S^1)_+  \to \bfE \wedge S^1$ and the retraction $\bfE \wedge (S^1)_+ \to \bfE$.
There is a weak equivalence $\bfN_{\pm}\bfE \vee \bfE \to \bfZ_{\pm}\bfE$. We conclude from 
Lemma~\ref{lem:c_contracted_and_vee} that $\bfN_{\pm}\bfE$, $\bfB_r\bfE$ and $\bfB\bfE$ are $c$-contracted.

By construction we have the homotopy cofibration sequence $\bfB_r\bfE \to \bfZ\bfE \to \bfL\bfE$.
It induces a long exact sequence of homotopy groups. The existence
of the retractions $\rho_i$ imply that it breaks up into short split exact sequences
of transformations of functors from $\addcat$ to the category of abelian groups
\[
0 \to \pi_i(\bfB_r\bfE) \to \pi_i(\bfZ\bfE) \to \pi_i(\bfL\bfE) \to 0.
\]
We obtain a commutative diagram with short split exact rows
as vertical arrows, where the retractions from the middle term to the left term are also compatible
with the vertical maps.
\[
\xymatrix{0 \ar[r] 
&
\pi_i(\bfZ_{\pm}\bfB_r\bfE) \ar[r] \ar[d]^{\pi_i(\bfev_0^{\pm}(\bfB_r\bfE))}
&
\pi_i(\bfZ_{\pm}\bfZ\bfE) \ar[r] \ar[d]^{\pi_i(\bfev_0^{\pm}(\bfZ\bfE))}
&
\pi_i(\bfZ_{\pm}\bfL\bfE) \ar[d]^{\pi_i(\bfev_0^{\pm}(\bfL\bfE))} \ar[r]
& 0
\\
0 \ar[r] 
&
\pi_i(\bfB_r\bfE) \ar[r] 
&
\pi_i(\bfZ\bfE) \ar[r] 
&
\pi_i(\bfL\bfE) \ar[r]
& 0
}
\]
Since we have the isomorphism
\[
\pi_i(\bfb_{\pm}) \oplus \pi_i(\bfi_+) \colon \pi_i(\bfN_{\pm}\bfE) \oplus  \pi_i(\bfE)  
\xrightarrow{\cong} \pi_i(\bfZ_{\pm}\bfE),
\]
and $\pi_i(\bfev^{\pm}_0) \circ \pi_i(\bfi_+) = \id$,  we obtain the short split exact sequence
\[
0 \to \pi_i(\bfN_{\pm}\bfB_r\bfE) \to \pi_i(\bfN_{\pm}\bfZ\bfE) \to \pi_i(\bfN_{\pm}\bfL\bfE) \to 0.
\]
We have the obvious short split exact sequences
\[
0 \to \pi_i\bigl(\bfB_r\bfE \wedge (S^1)_+\bigr) \to \pi_i\bigl(\bfZ\bfE\wedge (S^1)_+ \bigr) 
\to \pi_i\bigl(\bfL\bfE\wedge (S^1)_+ \bigr) \to 0.
\]
Taking direct sums shows that we obtain  short split exact sequences
\[
0 \to \pi_i(\bfB\bfB_r\bfE) \to \pi_i(\bfB\bfZ\bfE) \to \pi_i(\bfB\bfL\bfE) \to 0,
\]
and 
\[
0 \to \pi_i(\bfB_r\bfB_r\bfE) \to \pi_i(\bfB_r\bfZ\bfE) \to \pi_i(\bfB_r\bfL\bfE) \to 0.
\]

Thus we obtain for all $i \in \IZ$ a commutative diagram with exact rows
\[\xymatrix{
0 \ar[r]
&
\pi_i(\bfB\bfB_r\bfE) \ar[r] \ar[d]^{\pi_i(\bfBHS(\bfB_r\bfE))}
&
\pi_i(\bfB\bfZ\bfE) \ar[r] \ar[d]^{\pi_i(\bfBHS(\bfZ\bfE))}
&
\pi_i(\bfB\bfL\bfE) \ar[r] \ar[d]^{\pi_i(\bfBHS(\bfL\bfE))}
& 
0
\\
0 \ar[r]
&
\pi_i(\bfZ\bfB_r\bfE) \ar[r] 
&
\pi_i(\bfZ\bfZ\bfE) \ar[r] 
&
\pi_i(\bfZ\bfL\bfE) \ar[r] 
& 
0
}
\]

Since $\pi_i(\bfBHS(\bfB_r\bfE))$ and $\pi_i(\bfBHS(\bfZ\bfE))$
are isomorphisms for $i \ge -c+1$, the same is true for $\pi_i(\bfBHS(\bfL\bfE))$
by the Five-Lemma. 

The following diagram commutes and has  exact rows
\begin{eqnarray}
& \xymatrix{
0 \ar[r]
&
\pi_i(\bfB_r\bfB_r\bfE) \ar[r] \ar[d]^{\pi_i(\bfBHS_r(\bfB_r\bfE))}
&
\pi_i(\bfB_r\bfZ\bfE) \ar[r] \ar[d]^{\pi_i(\bfBHS_r(\bfZ\bfE))}
&
\pi_i(\bfB_r\bfL\bfE) \ar[r] \ar[d]^{\pi_i(\bfBHS_r(\bfL\bfE))}
& 
0
\\
0 \ar[r]
&
\pi_i(\bfZ\bfB_r\bfE) \ar[r] 
&
\pi_i(\bfZ\bfZ\bfE) \ar[r] 
&
\pi_i(\bfZ\bfL\bfE) \ar[r] 
& 
0
}
&
\label{diagram_1}
\end{eqnarray}
The first two vertical arrows are split injective. We claim that the retractions fit into the
following commutative square
\begin{eqnarray}
& 
\xymatrix{\pi_i(\bfZ\bfB_r\bfE) \ar[d]_{\rho_i(\bfB_r\bfE)}  \ar[r]
&
\pi_i(\bfZ\bfZ\bfE) \ar[d]^{\rho_i(\bfZ\bfE)}
\\
\pi_i(\bfB_r\bfB_r\bfE) \ar[r]
&
\pi_i(\bfB_r\bfZ\bfE) 
}
&
\label{diagram_2}
\end{eqnarray}
This follows from the fact that we have the commutative diagram with isomorphisms as horizontal arrows
\[
\xymatrix@!C=20em{
\pi_i(\bfZ\bfE) \oplus \pi_i(\bfZ\bfN_+\bfE) \oplus \pi_i(\bfZ\bfN_-\bfE) 
\ar[r]^-{\pi_i(\bfZ\bfi_0) \oplus \pi_i(\bfZ\bfb_+) \oplus \pi_i(\bfZ\bfb_-)}_-{\cong} 
\ar[d]_{\rho_i(\bfE) \oplus \rho_i(\bfN_+\bfE) \oplus \rho_i(\bfN_-\bfE)}
&
 \pi_1(\bfZ \bfB_r \bfE) \ar[d]^{\rho_i(\bfB_r\bfE)}
\\
\pi_i(\bfB_r\bfE) \oplus \pi_i(\bfB_r\bfN_+\bfE) \oplus \pi_i(\bfB_r\bfN_-\bfE) 
\ar[r]^-{\pi_i(\bfB_r\bfi_0) \oplus \pi_i(\bfB_r\bfb_+) \oplus \pi_i(\bfB_r\bfb_-)}_-{\cong} 
&
\pi_i(\bfB_r \bfB_r \bfE)
}
\]
and the following commutative diagrams
\[
\xymatrix@!C=7em{\pi_i(\bfZ\bfN_{\pm}\bfE) \ar[d]_{\rho_i(\bfN_{\pm}\bfE)}  \ar[r]^{\pi_i(\bfZ\bfb_{\pm})}
&
\pi_i(\bfZ\bfZ\bfE) \ar[d]^{\rho_i(\bfZ\bfE)}
\\
\pi_i(\bfB_r\bfN_{\pm}\bfE) \ar[r]_{\pi_i(\bfB_r\bfb_{\pm})}
&
\pi_i(\bfB_r\bfZ\bfE) 
}
\]
and
\[
\xymatrix@!C=6em{\pi_i(\bfZ\bfE) \ar[d]_{\rho_i(\bfE)}  \ar[r]^{\pi_i(\bfZ\bfi_0)}
&
\pi_i(\bfZ\bfZ\bfE) \ar[d]^{\rho_i(\bfZ\bfE)}
\\
\pi_i(\bfB_r\bfE) \ar[r]_{\pi_i(\bfB_r\bfi_0)}
&
\pi_i(\bfB_r\bfZ\bfE) 
}
\] 

The two diagrams~\eqref{diagram_1} and~\eqref{diagram_2}  imply that 
$\pi_i(\bfBHS_r(\bfL\bfE)) \colon \pi_i(\bfB_r\bfL\bfE) \to \pi_i(\bfZ\bfL\bfE)$ 
is split injective for all $i \in \IZ$.  This finishes the proof
that $\bfL\bfE$ is $c$-contracted.

We have the following diagram which has  homotopy cofibration sequences 
as vertical arrows and which commutes up to homotopy.
\[
\xymatrix{\bfB_r\bfE \ar[r] \ar[d]^{\id}
&
\bfB\bfE \ar[d]^{\bfBHS}  \ar[r] 
& 
\Sigma \bfE \ar[d]^{\bfs'}
\\
\bfB_r\bfE \ar[r]^{\bfBHS_r} 
&
\bfZ\bfE \ar[r]
&
\bfL\bfE
}
\]
The long exact homotopy sequences associated to the rows and the fact that
$\pi_i(\bfBHS_r) \colon \pi_i(\bfB_r\bfE) \to \pi_i(\bfZ\bfE)$ is split injective for 
$i \in \IZ$ imply that we obtain for all $i \in \IZ$ a commutative diagram with exact rows.
\[
\xymatrix{0 \ar[r]
&
\pi_i(\bfB_r\bfE) \ar[r] \ar[d]^{\id}
&
\pi_i(\bfB\bfE) \ar[r] \ar[d]^{\pi_i(\bfBHS)}
&
\pi_i(\Sigma\bfE) \ar[r] \ar[d]^{\bfs'} 
& 
0
\\
0 \ar[r]
&
\pi_i(\bfB_r\bfE) \ar[r] 
&
\pi_i(\bfZ\bfE) \ar[r] 
&
\pi_i(\bfL\bfE) \ar[r] 
& 
0
}
\]
Since $\pi_i(\bfBHS)$ is bijective for $i \ge -c+1$ by assumption, the same is true for $\pi_i(\bfs')$.
This finishes the proof of Lemma~\ref{E-n-contracted_implies_LE_is_n-contracted}.
\end{proof}


\typeout{-----------------------  Section 3: The delooping construction ------------------------}

\section{The delooping construction}
\label{sec:The_delooping_construction}

Let $\bfE \colon \addcat \to \Spectra$ be a covariant functor satisfying Condition~\ref{con:condition_bfu}.
Next we define inductively a sequence of spectra 
\begin{eqnarray}
& (\bfE[n])_{n \ge 0} &
\label{sequence_bfE[n]}
\end{eqnarray}
together with maps of spectra
\begin{eqnarray}
\bfs[n] \colon \bfE[n] \to  \bfE[n+1]  & & \text{for} \; n \ge 0.
\label{bfs[n]}
\end{eqnarray}
We define $\bfE[0]$ to be $\bfE$. In the induction step we have to explain how we construct
$\bfE[n + 1]$ and $\bfs[n]$ provided that we have defined  $\bfE[n]$. Define
$\bfE[n+1] = \Omega \bfL \bfE[n]$ and let $\bfs[n]$ be the map 
$\bfs \colon \bfE[n] \to \Omega \bfL \bfE[n]$ associated to $\bfE[n]$.

\begin{definition}[Delooping ${\bfE[\infty]}$]
\label{def:E[infty]}
Define the \emph{delooping} $\bfE[\infty]$ of   $\bfE$ to be the homotopy colimit of the sequence
\[
\bfE = \bfE[0] \xrightarrow{\bfs[0]} \bfE[1] \xrightarrow{\bfs[1]}\bfE[2] \xrightarrow{\bfs[2]} \cdots.
\]
Define the map of spectra
\[
\bfd \colon \bfE  \to \bfE[\infty]
\]
to be the zero-th structure map of the homotopy colimit.
\end{definition}

\begin{theorem}[Main property of the delooping construction]
\label{the:Main_property_of_the_delooping_construction}
Fix an integer $c$. Suppose that $\bfE$ is $c$-contracted. Then
\begin{enumerate}

\item \label{the:Main_property_of_the_delooping_construction:pi_i(bfd)} 
The map $\pi_i(\bfd) \colon \pi_i(\bfE) \to \pi_i(\bfE[\infty])$ is bijective for $i \ge -c$;

\item \label{the:Main_property_of_the_delooping_construction:E[infty]} 
$\bfE[\infty]$ is $\infty$-contracted;

\item \label{the:Main_property_of_the_delooping_construction:d} $\bfE$ is $\infty$-contracted
if and only if  $\bfd \colon \bfE \to \bfE[\infty]$ is a weak equivalence.
\end{enumerate}
\end{theorem}
\begin{proof}~\ref{the:Main_property_of_the_delooping_construction:pi_i(bfd)} 
  This follows from the fact that $\colim_{n \to \infty} \pi_i(\bfE[n]) = \pi_i(\bfE[\infty])$
  and the conclusion of 
  Lemma~\ref{E-n-contracted_implies_LE_is_n-contracted} that 
  $\pi_i(\bfs[n]) \colon \pi_i(\bfE[n]) \to \pi_i[\bfE[n+1])$ is bijective for $i \ge c$.
  \\[2mm]~\ref{the:Main_property_of_the_delooping_construction:E[infty]} over $n$ we
  conclude from Lemma~\ref{E-n-contracted_implies_LE_is_n-contracted} that $\bfE[n]$ is
  $(n+c)$-contracted for $n \ge 0$.  Obviously $\hocolim$ and $\bfZ^+$ commute as well as
  $\hocolim$ and $\bfZ$.  Hence $\hocolim$ and $\bfN_{\pm}$ commute up to weak equivalence,
  since $\hocolim$ is compatible with $\vee$ up to weak equivalence and we have a natural equivalence 
  $\bfE \vee  \bfN_{\pm}\bfE \to \bfZ_{\pm}\bfE$. This implies that $\hocolim$ and $\bfB$ commute up to weak
  equivalence.  Obviously $\hocolim$ commutes with $- \wedge (S^1)_+$.
Hence we obtain for each $i \in \IZ$ the following commutative diagram
  with isomorphisms as horizontal maps
  \[
\xymatrix{\colim_{n \to \infty}  \pi_i(\bfB\bfE[n]) \ar[r]^-{\cong} \ar[d]_{\colim_{n \to \infty} \pi_i(\bfBHS(\bfE[n]))}
& \pi_i(\bfB \bfE[\infty]) \ar[d]^{\pi_i(\bfBHS(\bfE[\infty])}
\\
\colim_{n \to \infty}  \pi_i(\bfZ\bfE[n]) \ar[r]^-{\cong} 
& \pi_i(\bfZ \bfE[\infty]) 
}
\]
Since $\bfE[n]$ is $(n+c)$-contracted, the left arrow and hence the right arrow are
isomorphisms for all $i \in \IZ$.
\\[2mm]~\ref{the:Main_property_of_the_delooping_construction:d} If $\bfd$ is a weak
equivalence, then $\bfBHS \colon \bfB\bfE \to \bfZ\bfE$ is a weak equivalence by
assertion~\ref{the:Main_property_of_the_delooping_construction:E[infty]} and the fact that
the following diagram commutes and has weak equivalences as horizontal arrows
\[
\xymatrix{\bfB\bfE \ar[r]^-{\bfB\bfd}_-{\simeq} \ar[d]_{\bfBHS(\bfE)} 
& \bfB \bfE[\infty] \ar[d]^{\bfBHS(\bfE[\infty])} 
\\
\bfZ\bfE \ar[r]_-{\bfZ\bfd}^-{\simeq} &
\bfZ\bfE[\infty]
}
\]
Suppose that $\bfBHS \colon \bfB\bfE \to \bfZ\bfE$ is a weak equivalence.
Then $\bfE$ is $c$-contracted for all $c \in \IZ$. Because of 
Lemma~\ref{E-n-contracted_implies_LE_is_n-contracted} $\bfE[n]$ is $c$-contracted for all $c \in \IZ$
and $\pi_i(\bfs[n]) \colon \pi_i(\bfZ\bfE[n]) \to \pi_i(\bfZ\bfE[n+1])$ is bijective for all $i \in \IZ$ and $n \ge 0$.
This implies that $\pi_i(\bfd)$ is bijective for all $i \in \IZ$.
\end{proof}

\begin{remark}[Retraction needed in all degrees]
  \label{rem:retraction_needed_in_all_degrees}
  One needs the retractions $\rho_i$ appearing in Definition~\ref{def:c-contracted_functor} in
  each degree $i \in \IZ$ and not only in degree $-c$.  The point is that one has a
  $c$-contracted functor $\bfE$ and wants to prove that $\bfE[n]$ is
  $(c+n)$-contracted. For this purpose one needs the retraction to split up certain long
  exact sequence into pieces in dimensions $i \ge -c$ to verify bijectivity on $\pi_i$ for
  $i \ge -c$, but also in degree $i = -c-1$, to construct the retraction for $\bfE[1]$ in
  degree $-c-1$.  For this purpose one needs the retraction for $\bfE$ also in degree
  $-c-2$.  In order to be able to iterate this construction, namely, to pass from
  $\bfE[1]$ to $\bfE[2]$, we have the same problem, the retraction for $\bfE[1]$ must also
  be present in degree $-c-3$. Hence we need a priori the retractions for $\bfE$ also in degree
  $-c-3$. This argument goes on and on and forces us to require the retractions in all
  degrees.
  
  One needs retractions rather than injective maps in the definition of $c$-contract\-ed.  
  Injectivity would suffice
  to  reduce the long exact sequences obtained after taking homotopy groups to short exact sequences  and
  most of the arguments involve the Five-Lemma where only  short exact but not split short
  exact is needed. However, at one point we want to argue for a commutative diagram with
  exact rows of abelian groups
\[
\xymatrix{0 \ar[r] & A_0 \ar[r] \ar[d] &  A_1 \ar[r]  \ar[d] & A_2 \ar[r]   \ar[d]& 0
\\
0 \ar[r] & B_0 \ar[r] &  B_1 \ar[r] & B_2 \ar[r] & 0
}
\]
that the third vertical arrow admits a retraction if the first and the second arrow admit retractions compatible 
with the two first horizontal arrows. This is true. But  the corresponding statement is wrong
if we  replace ``admitting a retraction'' by ``injective''. 
\end{remark}

\begin{lemma} \label{lem:alpha_bijective} Suppose that the covariant functors $\bfE, \bfF
  \colon \addcat \to \Spectra$ satisfy Condition~\ref{con:condition_bfu} and are
  $\infty$-contracted. Let $\bff \colon \bfE \to \bfF$ be a compatible transformation.  Suppose that
  there exists an integer $N$ such that $\pi_i(\bff(\cala))$ is bijective for all $i \ge
  N$ and all objects $\cala$ in $\addcat$.

  Then $\bff\colon \bfE \to \bfF$ is a weak equivalence.
 \end{lemma}
  \begin{proof} 
     We show by induction over $i$ that $\pi_i(\bff(\cala))$ is bijective for $i = N, N-1, N-2$
    and all objects objects $\cala$ in $\addcat$.
   The induction beginning $i = N$ is trivial, the induction step from $i$ to $i-1$ done as follows.

   We have the following commutative diagram whose horizontal arrows come from
   the Bass-Heller-Swan maps and hence are bijective by assumption
    \[
    \xymatrix@!C= 14em{\pi_{i-1}(\bfE(\cala)) \oplus \pi_{i}(\bfE(\cala)) \oplus \pi_{i}(\bfN_+\bfE(\cala)) \oplus \pi_{i}(\bfN_-\bfE(\cala)) 
     \ar[r]^-{\cong} \ar[d]_{\pi_{i-1}(\bff(\cala)) \oplus \pi_i(\bff(\cala)) \oplus \pi_i(\bfN_+\bff(\cala)) \oplus \pi_i(\bfN_-\bff(\cala))}
     &\pi_i(\bfZ\bfE(\cala)) \ar[d]^{\pi_i(\bfZ\bff(\cala))}
     \\
     \pi_{i-1}(\bfF(\cala)) \oplus \pi_{i}(\bfF(\cala)) \oplus \pi_{i}(\bfN_+\bfF(\cala)) \oplus \pi_{i}(\bfN_-\bfF(\cala)) 
     \ar[r]_-{\cong}
     &\pi_i(\bfZ\bfF(\cala)) 
     }
     \]
By  the  induction hypothesis $\pi_i(\bfZ\bff(\cala))$ is bijective. Hence $\pi_{i-1}(\bff(\cala))$ is bijective since it is 
a direct summand in $\pi_i(\bfZ\bff(\cala))$.
\end{proof}

Theorem~\ref{the:Main_property_of_the_delooping_construction} and  Lemma~\ref{lem:alpha_bijective} imply

\begin{corollary}
  \label{cor:turning_a_transformation_into_a_weak_equivalence}
  Suppose that the functors $\bfE, \bfF \colon \addcat \to \Spectra$ satisfy
  Condition~\ref{con:condition_bfu}. Suppose that $\bfE$ and $\bfF$ are  $c$-contracted 
  for some integer $c$. Let $\bff \colon \bfE \to \bfF$ be a compatible transformation.
  Suppose that there exists an integer $N$  such that
  $\pi_i(\bff) \colon \pi_i(\bfE) \to \pi_i(\bfF)$ is bijective for $i \ge N$.

  Then the following diagram commutes
  \[
  \xymatrix@!C= 5em{\bfE \ar[d]_{\bff} \ar[r]^{\bfd(\bfE)} 
   & 
    \bfE[\infty]
    \ar[d]^{\bff[\infty]}_{\simeq}
    \\
    \bfF \ar[r]_{\bfd(\bfF)}
    & \bfF[\infty]
    }
  \]
  and the right vertical arrow is a weak equivalences.
  \end{corollary}

  \begin{remark}[Universal property of the delooping construction in the homotopy category]
  \label{rem:Universal_property_of_the_delooping_construction_in_the_homotopy_category}
  Suppose that the covariant functors $\bfE, \bfF \colon \addcat \to \Spectra$ satisfy
  Condition~\ref{con:condition_bfu}. Suppose that $\bfE$ is $c$-contracted for some
  integer $c$, and let $\bff \colon \bfE \to \bfF$ be a compatible transformation to an
  $\infty$-contracted functor.

  Then, in the homotopy category (of functors $\addcat\to\Spectra$), the transformation $\bff$ 
  factors uniquely through $\bfd(\bfE)\colon \bfE\to \bfE[\infty]$:
\[\bff = \bfd(\bfF)\inv \circ \bff[\infty] \circ \bfd(\bfE).\]
   \end{remark}


\typeout{-----------------------  Section 4: Delooping algebraic K-theory of additive categories------------------------}

\section{Delooping algebraic K-theory of additive categories}
\label{sec:Delooping_algebraic_K-theory_of_additive_categories}

Now we treat our main example for $\bfE$, the functor
$\bfK \colon \addcat \to \Spectra$ that  assigns to an additive category $\cala$
the  connective $K$-theory spectrum $\bfK(\cala)$ of  $\cala$.

\begin{definition}[Non-connective algebraic $K$-theory spectrum {$\bfKinfty$}]
\label{def:non-connective_algebraic_K-theory_spectrum_bfK[infty]}
We call the functor 
\[
\bfKinfty := \bfK[\infty] \colon \addcat \to \Spectra
\]
associated to $\bfK  \colon \addcat \to \Spectra$ in Definition~\ref{def:E[infty]} the 
\emph{non-connective algebraic $K$-theory functor}.

If $\cala$ is an additive category, then $K_i(\cala) := \pi_i(\bfKinfty(\cala))$ 
is the $i$-th algebraic $K$-group of $\cala$ for $i \in \IZ$.
\end{definition}

Notice that by Lemma~\ref{lem:alpha_bijective} we could have as well defined $\bfKinfty$
to be $\bfKidem[\infty]$. In particular, by 
Theorem~\ref{the:Bass-Heller-Swan_Theorem_for_bfKIdem} and 
Theorem~\ref{the:Main_property_of_the_delooping_construction}~\ref{the:Main_property_of_the_delooping_construction:pi_i(bfd)},
Definition~\ref{def:non-connective_algebraic_K-theory_spectrum_bfK[infty]} extends the
previous definition $K_i(\cala) := \pi_i\bigl(\bfKidem(\cala)\bigr)$ for $i \ge 0$
of~\eqref{K_i(cala)_i_ge_0}.

We conclude from Theorem~\ref{the:Bass-Heller-Swan_Theorem_for_bfK}
and Theorem~\ref{the:Main_property_of_the_delooping_construction}%
~\ref{the:Main_property_of_the_delooping_construction:E[infty]}

\begin{theorem}[Bass-Heller-Swan-Theorem for $\bfKinfty$]
\label{the:Bass-Heller-Swan-Theorem_for_non-connective_algebraic_K-theory}
The Bass-Heller-Swan transformation
\[
\bfBHS \colon \bfKinfty
 \wedge (S^1)_+ \vee \bfN_+\bfKinfty \vee \bfN_-\bfKinfty \xrightarrow{\simeq} \bfZ\bfKinfty
\]
is a weak equivalence. 

In particular we get for every $i \in \IZ$ and every additive category $\cala$ an in
$\cala$-natural isomorphism
\[
K_{i-1}(\cala) \oplus K_i(\cala) \oplus N\!K_i(\cala[t]) \oplus N\!K_i(\cala[t^{-1}]) \xrightarrow{\cong} K_i(\cala[t,t^{-1}]),
\]
where $N\!K_i(\cala[t^{\pm}])$ is defined as the kernel of $K_i(\ev_0^{\pm}) \colon K_i(\cala[t^{\pm}]) \to K_i(\cala)$.
\end{theorem}

We will extend Theorem~\ref{the:Bass-Heller-Swan-Theorem_for_non-connective_algebraic_K-theory}
later to the twisted case.

\begin{remark}[Fundamental sequence]
\label{rem:fundamental_sequence}
Theorem~\ref{the:Bass-Heller-Swan-Theorem_for_non-connective_algebraic_K-theory} 
is equivalent to the statement that there exists for each $i \in \IZ$ a fundamental sequence of algebraic $K$-groups
\begin{multline*}
0 \to K_i(\cala) \xrightarrow{(K_i(i_+), -K_i(i_-))} K_i(\cala[t]) \oplus K_i(\cala[t^{-1}]) 
\\
\xrightarrow{K_i(j_+) \oplus K_i(j_-)} 
K_i(\cala[t,t^{-1}]) \xrightarrow{\partial_i} K_{i-1}(\cala) \to 0
\end{multline*}
which comes with a splitting $s_{i-1} \colon K_{i-1}(\cala) \to K_i(\cala[t,t^{-1}])$ of $\partial_i$, natural in $\cala$.
\end{remark}

\begin{remark}[Identification with the original negative $K$-groups]
\label{rem:Identification_with_the_original_negative_K-groups}
Bass~\cite[page~466 and page~462]{Bass(1968)} (see also~\cite[Chapter~3, Section~3]{Rosenberg(1994)})  
defines negative $K$-groups $K_i(\cala)$ for $i = -1,-2, \ldots$ inductively
by putting 
\[
K_{i-1}(\cala) := \cok\left(K_i(j_+) \oplus K_i(j_-) \oplus K_i(\cala[t]) \oplus K_i(\cala[t^{-1}])  \to K_i(\cala[t,t^{-1}])\right).
\]
We conclude from Remark~\ref{rem:fundamental_sequence}
that the negative $K$-groups of
Definition~\ref{def:non-connective_algebraic_K-theory_spectrum_bfK[infty]}
are naturally isomorphic to the negative $K$-groups defined by Bass. 
\end{remark}

\begin{remark}[Identification with the Pedersen-Weibel construction]
\label{rem:identification_with_Pedersen-Weibel}
Pedersen-Weibel~\cite{Pedersen-Weibel(1985)} construct another transformation $\bfK_{\PW} \colon \addcat \to \Spectra$
which models negative algebraic $K$-theory. We conclude from 
Corollary~\ref{cor:turning_a_transformation_into_a_weak_equivalence}
that there exists  weak equivalences
\[
\bfKinfty \xrightarrow{\simeq} \bfK_{\PW}[\infty] \xleftarrow{\simeq} \bfK_{\PW}
\]
since there is a natural map $\bfK \to \bfK_{\PW}$ inducing on  $\pi_i$ bijections for $i \ge 1$ and
the Bass-Heller-Swan map for $\bfK_{\PW}$ is a weak equivalence as 
$\pi_i(\bfK_{\PW}[\infty])$ agrees in a natural way 
with the $i$-th homotopy groups of the connective $K$-theory for $i \ge 1$ and
with the negative $K$-groups of Bass for $i \le 0$, see~\cite[Theorem~A]{Pedersen-Weibel(1985)}.
\end{remark}


\typeout{---------------  Section 5: Compatibility of the delooping construction with homotopy colimits ----------------------}

\section{Compatibility with colimits}
\label{sec:Compatibility_of_the_delooping_construction_with_homotopy_colimits}

Let $\calj$ be a small category, not necessarily filtered or finite. Recall the notation
\[
\funcc(\addcat,\Spectra)
\]
from Definition~\ref{def:funcc}. Consider a $\calj$-diagram in
$\funcc(\addcat,\Spectra)$, i.e., a covariant functor $\bfE \colon \calj \to
\funcc(\addcat,\Spectra)$. There is the functor homotopy colimit
\[
\hocolim_{\calj} \colon \funcc(\calj,\Spectra) \to \Spectra
\]
which sends a $\calj$-diagram of spectra, i.e., a covariant functor $\calj \to \Spectra$,
to its homotopy colimit.  As a consequence of Lemma~\ref{lem:inheritance_of_condition_under_general_constructions}, 
it induces a functor, denoted by the same symbol
\[
\hocolim_{\calj} \colon \func\bigl(\calj,\funcc(\addcat,\Spectra)\bigr) \to
\funcc(\addcat,\Spectra),
\]
that sends a $\calj$-diagram $(\bfE(j))_{j \in \calj}$ to the functor $\addcat \to
\Spectra$ which assigns to an additive category $\cala$ the spectrum $\hocolim_{\calj}
\bfE(j)(\cala)$.

\begin{theorem}[Compatibility of the delooping construction with homotopy colimits]
\label{the:Compatibility_of_the_delooping_construction_with_homotopy_colimits}
Given a $\calj$-diagram $\bfE$ in $\func_c(\addcat,\Spectra)$, there is a morphism in
$\funcc(\addcat,\Spectra)$, natural in $\bfE$,
\[
\gamma(\bfE) \colon \hocolim_{\calj} \bigl(\bfE(j)[\infty]\bigr) \xrightarrow{\simeq}
\bigl(\hocolim_{\calj} \bfE(j)\bigr)[\infty]
\]
that is a weak equivalence, i.e., its evaluation at any object in $\addcat$ is a weak
homotopy equivalence of spectra.
\end{theorem}


The proof uses some well-known properties of homotopy colimits of spectra, 
which we record here for the reader's convenience.

\begin{lemma}\label{lem:diagrams_of_spectra}
  Let $\bfE$ and $\bfF$ be $\calj$-diagrams of spectra and let $\bff \colon \bfE \to \bfF$
  be a morphism between them.
\begin{enumerate}

\item \label{lem:diagrams_of_spectra:vee}
The canonical map
\[\bigl(\hocolim_{\calj} \bfE\bigr)  \vee \bigl(\hocolim_{\calj} \bfF \bigr) 
\xrightarrow{\simeq} \hocolim_{\calj} \bigl(\bfE \vee \bfF \bigr)
\]
is an isomorphism;

\item \label{lem:diagrams_of_spectra:smash}
If $Y$ is a pointed space, then we obtain an isomorphism, natural in $\bfE$,
\[
\hocolim_{\calj} \bigl(\bfE \wedge Y\bigr) \xrightarrow{\simeq} \bigl(\hocolim_{\calj} \bfE \bigr) \wedge Y;
\]

\item \label{lem:diagrams_of_spectra:Omega}
There is  a weak homotopy equivalence, natural in $\bfE$,
\[
\hocolim_{\calj} \bigl(\Omega\bfE\bigr) \xrightarrow{\simeq} \Omega \bigl(\hocolim_{\calj} \bfE \bigr);
\]

\item \label{lem:diagrams_of_spectra:commute}
If $\calk$ is another small category and we have a $\bfJ \times \bfK$ diagram $\bfE$ of spectra.
Then we have isomorphisms of spectra, natural in $\bfE$,
\[
\hocolim_{\calj} \bigl(\hocolim_{\calk} \bfE \bigr) \xrightarrow{\cong} 
\hocolim_{\calj \times \calk} \bfE \xleftarrow{\cong}
\hocolim_{\calk} \bigl(\hocolim_{\calj} \bfE \bigr);
\]

\item \label{lem:diagrams_of_spectra:homotopy_fibers}
Let $\hofib(\bff)$ and $\hocofib(\bff)$ respectively be the $\calj$-diagram 
of spectra which assigns to an object $j$ in $\calj$
the homotopy fiber  and homotopy cofiber respectively 
 of $\bff(j) \colon \bfE(j) \to \bfF(j)$. 

Then there are   weak homotopy equivalences, natural in $\bff$,
\begin{eqnarray*}
\hocolim_{\calj}  \hocofib(\bff) & \xrightarrow{\simeq} & \hocofib\bigl(\hocolim_{\calj} \bff \bigr);
\\
\hocolim_{\calj}  \hofib(\bff) & \xrightarrow{\simeq} & \hofib\bigl(\hocolim_{\calj} \bff \bigr);
\end{eqnarray*}

\item \label{lem:diagrams_of_spectra:pi_and_filtered}
If $\calj$ is filtered, i.e., for any two objects $j_0$ and $j_1$ there exists a morphism  $u \colon j \to j'$ in $\calj$
such that there exists morphisms from both $j_0$ and $j_1$ to $j$ and for any two morphisms
$u_0 \colon j_0 \to j$ and $u_1 \colon j_1 \to j$ we have $u \circ j_0 = u \circ u_1$. Then the canonical map
\[
\colim_{\calj} \pi_i(\bfE(j)) \xrightarrow{\cong} \pi_i\bigl(\hocolim_{\calj} \bfE(j)\bigr)
\]
is bijective for all $i \in \IZ$.
\end{enumerate}
\end{lemma}

\begin{proof}[Proof of Theorem~\ref{the:Compatibility_of_the_delooping_construction_with_homotopy_colimits}]
Let $\bfE$ be a $\calj$-diagram in $\funcc(\addcat,\Spectra)$.
We have by definition the equalities
\begin{eqnarray*}
\bfZ \bigl(\hocolim_{\calj} \bfE\bigr) & = & \hocolim_{\calj} (\bfZ\bfE);
\\
\bfZ_{\pm} \bigl(\hocolim_{\calj} \bfE\bigr) & = & \hocolim_{\calj} (\bfZ_{\pm}\bfE).
\end{eqnarray*}
We obtain from Lemma~\ref{lem:diagrams_of_spectra}~\ref{lem:diagrams_of_spectra:smash}
and~\ref{lem:diagrams_of_spectra:homotopy_fibers}
natural weak homotopy equivalences
\begin{eqnarray*}
\hocolim_{\calj} (N_{\pm}\bfE)
& \xrightarrow{\simeq} &
\bfN_{\pm} \bigl(\hocolim_{\calj} \bfE\bigr);
\\
\bigl(\hocolim_{\calj} \bigl(\bfE  \wedge (S^1)_+\bigr)
& \xrightarrow{\simeq} &
\bigl(\hocolim_{\calj} \bfE \bigr) \wedge (S^1)_+,
\end{eqnarray*}
and thus by Lemma~\ref{lem:diagrams_of_spectra}~\ref{lem:diagrams_of_spectra:vee}
a natural weak homotopy equivalence
\begin{eqnarray*}
\hocolim_{\calj} (\bfB_r\bfE)
& \xrightarrow{\simeq} &
\bfB_r \bigl(\hocolim_{\calj} \bfE\bigr).
\end{eqnarray*}

As $\calj$ takes values in functors satisfying Condition~\ref{con:condition_bfu}, the maps
$\cala(j)\colon \bfE(j)\wedge (S^1)_+ \to \bfZ\bfE(j)$ are natural in $j\in \calj$.  In
this way $\bfL\bfE(j)$ also becomes a functor in $j$, and further applications of 
Lemma~\ref{lem:diagrams_of_spectra} show that the induced map
\[\hocolim_\calj \bfL\bfE \xrightarrow{\simeq} \bfL\bigl(\hocolim_\calj \bfE\bigr)\]
is a weak equivalence. We obtain a commutative diagram
\[
\xymatrix{
\hocolim_{\calj} \bfE \ar[d]^{\hocolim_{\calj}\bfs} \ar[r]^{\id}
& 
\bigl(\hocolim_{\calj}  \bfE \bigr) \ar[d]^{\bfs} 
\\
\hocolim_{\calj} \Omega \bfL\bfE   \ar[r]^{\simeq}
&
\Omega \bfL\bigl(\hocolim_{\calj}  \bfE\bigr)
}
\]
with weak homotopy equivalences as vertical arrows.

Iterating this construction leads to a  commutative diagram with weak homotopy equivalences
as vertical arrows.
\[
\xymatrix@!C= 30mm{
\hocolim_{\calj}  \bfE \ar[r]^-{\hocolim_{\calj} \bfs[0]}  \ar[d]^{\id}
& \hocolim_{\calj}  \bfE[1] \ar[r]^-{\hocolim_{\calj}  \bfs[1]} \ar[d]^{\simeq}
& \hocolim_{\calj}  \bfE[2] \ar[r]^-{\hocolim_{\calj} \bfs[2]} \ar[d]^{\simeq } 
&
\cdots
\\
\hocolim_{\calj}  \bfE \ar[r]^-{\bfs[0]}  
& \bigl(\hocolim_{\calj}  \bfE\bigr) [1] \ar[r]^-{\bfs[1]} 
& \bigl(\hocolim_{\calj}  \bfE\bigr) [2] \ar[r]^-{\bfs[2]} 
& \cdots
}
\]
After an application of Lemma~\ref{lem:diagrams_of_spectra}~\ref{lem:diagrams_of_spectra:commute} 
the induced map on homotopy colimits is the desired map $\gamma(\bfE)$.
\end{proof}


\typeout{----------  Section 6: The Bass-Heller-Swan decomposition for additive categories with automorphisms -----------------}

\section{The Bass-Heller-Swan decomposition for additive categories \\with automorphisms}
\label{sec:The_Bass-Heller-Swan_decomposition_for_additive_categories_with_automorphisms}

In~\cite[Theorem~0.4]{Lueck-Steimle(2013twisted_BHS)}  the following 
twisted Bass-Heller-Swan decomposition is proved for the connective $K$-theory spectrum.

\begin{theorem}[The twisted Bass-Heller-Swan decomposition for connective $K$-theory of
  additive categories]
  \label{the:BHS_decomposition_for_connective_K-theory}
  Let $\cala$ be an additive category which is idempotent complete. Let $\Phi \colon \cala \to \cala$ be an
  automorphism of additive categories. 

  \begin{enumerate}
  \item \label{the:BHS_decomposition_for_connective_K-theory:BHS-iso}

  Then there is a weak equivalence of  spectra, natural in $(\cala,\Phi)$,
  \[
  \bfa \vee \bfb_+\vee \bfb_- \colon \bfT_{\bfK(\Phi^{-1})} \vee
  \bfNK(\cala_{\Phi}[t]) \vee \bfNK(\cala_{\Phi}[t^{-1}]) \xrightarrow{\simeq}
  \bfK(\cala_{\Phi}[t,t^{-1}]);
  \]

  \item \label{the:BHS_decomposition_for_connective_K-theory:Nil}
  There exist  weak homotopy equivalences of spectra, natural in $(\cala,\Phi)$,
  \begin{eqnarray*}
   \Omega \bfNK(\cala_{\Phi}[t])  & \xleftarrow{\simeq} & \bfE(\cala,\Phi);
   \\
   \bfK(\cala) \vee  \bfE(\cala,\Phi)  & \xrightarrow{\simeq} &\bfK(\Nil(\cala,\Phi)). 
 \end{eqnarray*}
\end{enumerate}
\end{theorem}

Here $\bfT_{\bfK(\Phi\inv)}$ is the mapping torus of the map $\bfK(\Phi\inv)\colon
\bfK(\cala)\to \bfK(\cala)$, that is, the pushout of
\[\bfK(\cala)\wedge [0,1]_+ \xleftarrow{\id\wedge \operatorname{incl}} \bfK(\cala)\wedge \{0, 1\}_+ 
\xrightarrow{\bfK(\Phi\inv) \vee \id} \bfK(\cala).\] The spectrum
$\bfN\bfK(\cala_\Phi[t^{\pm1}])$ is by definition the homotopy fiber of the map
\[\bfK(\cala_\Phi[t^{\pm1}])\to \bfK(\cala)\]
induced by evaluation at 0. The category $\Nil(\cala,\Phi)$ is the exact category of
$\Phi$-nilpotent endomorphisms of $\cala$ whose objects are morphisms 
$f\colon \Phi(A)\to A$, with $A\in\ob(\cala)$ which are nilpotent in a suitable sense. For more details of the
construction of the spectra and maps appearing the result above, we refer
to~\cite[Theorem~0.1]{Lueck-Steimle(2013twisted_BHS)}.  In that paper it is also claimed that
Theorem~\ref{the:BHS_decomposition_for_connective_K-theory} implies by the delooping
construction of this paper in a formal manner the following non-connective version, where
the maps $\bfainfty$, $\bfbinfty_+$, and $\bfbinfty_-$ are defined completely analogous to
the maps $\bfa$, $\bfb_+$, $\bfb_-$, but now for $\bfK^{\infty}$ instead of $\bfK$.

\begin{theorem}[The twisted Bass-Heller-Swan decomposition for non-connective $K$-theory of
  additive categories]
  \label{the:BHS_decomposition_for_non-connective_K-theory}
  Let $\cala$ be an additive category. Let $\Phi \colon \cala \to \cala$ be an
  automorphism of additive categories. 

  \begin{enumerate}
  \item \label{the:BHS_decomposition_for_non-connective_K-theory:BHS-iso}

  There exists a weak homotopy equivalence of
  spectra, natural in $(\cala,\Phi)$, 
  \[
  \bfainfty \vee \bfbinfty_+\vee \bfbinfty_- \colon 
  \bfT_{\bfKinfty(\Phi^{-1})} \vee \bfNKinfty(\cala_{\Phi}[t]) \vee \bfNKinfty(\cala_{\Phi}[t^{-1}]) 
  \xrightarrow{\simeq}   \bfKinfty(\cala_{\Phi}[t,t^{-1}]);
  \]
  
  \item \label{the:BHS_decomposition_for_non-connective_K-theory:Nil}
  There exist  weak homotopy equivalences of spectra, natural in $(\cala,\Phi)$,
  \begin{eqnarray*}
   \Omega\bfNKinfty(\cala_{\Phi}[t]) & \xleftarrow{\simeq} & \bfE^{\infty}(\cala,\Phi);
   \\
   \bfKinfty(\cala) \vee  \bfE^{\infty} (\cala,\Phi)  & \xrightarrow{\simeq} & \bfKNilinfty(\cala,\Phi),
 \end{eqnarray*}
  where $\bfKNilinfty(\cala,\Phi)$ is a specific spectrum whose connective cover is the spectrum $\bfK\bigl(\Nil(\cala,\phi)\bigr)$.
 \end{enumerate}
\end{theorem}

To deduce Theorem~\ref{the:BHS_decomposition_for_non-connective_K-theory} from 
Theorem~\ref{the:BHS_decomposition_for_connective_K-theory}, we will think of all functors
appearing as functors in the pair $(\cala, \Phi)$ and apply a variation of the delooping
construction to these functors. In particular, the spectrum $\bfKNilinfty(\cala,\Phi)$ is
obtained by delooping the functor
\[
(\cala,\Phi)\mapsto \bfK(\Nil(\cala,\Phi)).
\] 
The next remark formalizes how to do deloop functors in $(\cala,\Phi)$.

\begin{remark}[Extension of the delooping construction to diagrams]
  \label{rem:Extension_of_the_delooping_construction_to_diagrams}
  Let $\calc$ be a fixed small category which will become an index category. 
  Let $\addcatc$ be the category of $\calc$-diagrams in $\addcat$, i.e., objects are
  covariant functors $\calc \to \addcat$ and morphisms are natural transformations of these.
  
 Our delooping construction can be extended from $\addcat$  to $\addcatc$ as follows,
 provided that the  obvious version of Condition~\ref{con:condition_bfu}
 which was originally stated for $\bfE \colon \addcat \to \Spectra$, holds now
 for $\bfE \colon \addcatc \to \Spectra$.
 
  The functors $z[t]$, $z[t^{-1}]$, $z[t,t^{-1}]$ from $\addcat$ to $\addcat$, extend to
  functors $z[t]^{\calc}$, $z[t^{-1}]^{\calc}$, $z[t,t^{-1}]^{\calc}$ $\addcatc \to \addcatc$ by composition.
  Analogously the natural transformations $i_0$, $i_{\pm}$, $j_{\pm}$ and $\ev_0^{\pm}$,
  originally defined for $\addcat$, do extend to natural transformations of functors
  $\addcatc \to \addcatc$.  Now the definitions of section~\ref{sec:The_Bass-Heller-Swan_map} make still sense if we start with a functor 
  $\bfE \colon \addcatc \to  \Spectra$ and end up with  functors 
  $\addcatc \to  \Spectra$, where it is to be understood that everything is compatible with
  Condition~\ref{con:condition_bfu}. Moreover, the notion of a $c$-contracted functor, the 
  construction of $\bfE[\infty]$,  Lemma~\ref{lem:c_contracted_and_vee}, 
  Theorem~\ref{the:Main_property_of_the_delooping_construction},
  Corollary~\ref{cor:turning_a_transformation_into_a_weak_equivalence} and 
  Theorem~\ref{the:Compatibility_of_the_delooping_construction_with_homotopy_colimits} 
  carry over word by word
  if we replace $\addcat$ by $\addcatc$ everywhere. From the definitions we also conclude:

\begin{lemma} \label{lem:commuting_[infty]_and_F} 
  Let $G \colon \addcatc \to  \addcat$ be a functor. 
  Suppose that there are natural isomorphisms in a commutative diagram
\[\xymatrix{
   G \circ z[t]^{\calc} \ar[d]^{G(j_+)} \ar[rr]^\cong && z[t] \circ G \ar[d]^{j_+}
  \\
  G \circ   z[t^{-1}]^{\calc}  \ar[rr]^\cong && z[t^{-1}] \circ G;
  \\
  G \circ z[t,t^{-1}]^{\calc} \ar[u]_{G(j_-)} \ar[rr]^\cong && z[t,t^{-1}] \circ G \ar[u]_{j_-}
}\]
  Let $\bfE \colon \addcat \to \Spectra$ be a functor
  respecting Condition~\ref{con:condition_bfu}.  

  Then $\bfE \colon \addcat
  \to \Spectra$ is $c$-contracted if and only if $\bfE \circ G \colon \addcat^{\calc} \to  \Spectra$ 
  is $c$-contracted, and we have a natural isomorphism 
    \[
    \bfE[\infty] \circ G \cong \bigl(\bfE \circ G\bigr)[\infty].
    \]
  \end{lemma}
\end{remark}

We will always be interested in the case where $\calc$ is the groupoid with one object 
  and $\IZ$ as automorphism group of this object.
  We will write $\addcatt$ for $\addcat^{\calc}$ in this case. Then
  objects in $\addcatt$ are pairs  $(\cala,\Phi)$  consisting of a small additive
  category $\cala$ and an automorphism $\Phi \colon \cala \xrightarrow{\cong} \cala$ and a
  morphism $F \colon (\cala_0,\Phi_0) \to (\cala_1,\Phi_1)$ is a functor of additive
  categories $F \colon \cala_0 \to \cala_1$ satisfying $\Phi_1 \circ F = F \circ   \Phi_0$.

  Our main examples for functors $G$ as appearing in Lemma~\ref{lem:commuting_[infty]_and_F} will be the functors
  \begin{align*}
    z_t[s^{\pm1}] \colon \addcatt &\to \Spectra,\quad (\cala,\Phi)\mapsto \cala_\Phi[s^{\pm1}],\\
    z_t[s,s\inv]\colon \addcatt & \to \Spectra, \quad (\cala,\Phi)\mapsto
    \cala_\Phi[s,s\inv].  \end{align*} (The subscript ``$t$'' stands for ``twisted'',
  since these functors are the obvious twisted generalizations of the functors
  $z[t^{\pm1}]$ and $z[t,t\inv]$ from Section~\ref{sec:The_Bass-Heller-Swan_map}, with the
  variable $t$ replaced by $s$ for the sake of readability.)

  Let us go back to the situation of Section~\ref{sec:The_Bass-Heller-Swan_map} where we
  were given a functor $\bfE\colon \addcat\to\Spectra$ satisfying 
  Condition~\ref{con:condition_bfu}. Replacing $z[t^{\pm1}]$ and $z[t,t\inv]$ by their twisted
  versions throughout, we may define the twisted versions
  \[
   \bfB^t\bfE, \bfN^t_\pm\bfE, \bfZ^t_\pm\bfE, \bfZ\bfE\colon \addcatt\to \Spectra
  \]
  of the corresponding functors appearing in 
  Section~\ref{sec:The_Bass-Heller-Swan_map}. The role of the functor $\bfE\wedge(S^1)_+$ is now
  taken by the functor
  \[\bfT^t\bfE(\cala,\Phi) = \bfT_{\bfE(\Phi\inv)}\]
  given by the mapping torus of the map $\bfE(\Phi\inv)\colon \bfE(\cala)\to  \bfE(\cala)$. 
  Condition~\ref{con:condition_bfu} implies in this setting that there is a
  natural transformation
  \[\bfa^t\colon \bfT^t\bfE \to \bfZ^t \bfE.\]
  (It is induced by the natural transformation
  \[\id_A\cdot t\colon \Phi\inv(A)\to A\]
  between the functors $\Phi\inv\circ i$ and $i$, where $i\colon \cala\to  \cala_\Phi[t,t\inv]$ 
  is the canonical inclusion.)

In these terms, the natural transformation from 
Theorem~\ref{the:BHS_decomposition_for_connective_K-theory}~\ref{the:BHS_decomposition_for_connective_K-theory:BHS-iso} 
is just given by the twisted version of the Bass-Heller-Swan map
\[
\bfBHS^t\colon \bfT^t\bfE \vee \bfN^t_+\bfE \vee \bfN^t_-\bfE \to \bfZ^t\bfE
\]
applied to $\bfE=\bfK$.

Next we want to apply the delooping construction to the Nil-groups.

\begin{lemma}\label{lem:nil_1_contracted}
The functor $(\cala,\Phi)\mapsto \bfK(\Nil(\cala,\Phi))$ is 1-contracted.
\end{lemma}

\begin{proof}
The functors 
\begin{align*}
 (\cala,\Phi) & \mapsto \bfK(\Nil(\Idem \cala,\Idem\Phi),\\
 (\cala,\Phi) & \mapsto \bfK(\Idem\cala)\vee \Omega\bfN\bfK((\Idem\cala)_{\Idem\Phi}[t\inv])
 \end{align*}
 are, by Theorem~\ref{the:BHS_decomposition_for_connective_K-theory}~\ref{the:BHS_decomposition_for_connective_K-theory:Nil},
 naturally equivalent. In the second functor, the first summand is 0-contracted by 
 Theorem~\ref{the:Bass-Heller-Swan_Theorem_for_bfKIdem}. The second summand is 0-contracted by
 Theorem~\ref{the:Bass-Heller-Swan_Theorem_for_bfK} and Lemma~\ref{lem:commuting_[infty]_and_F}, 
noticing that $\Omega$ decreases the degree of
 contraction by one. From 
Lemma~\ref{lem:c_contracted_and_vee}~\ref{lem:c_contracted_and_vee:contracted} it follows that
 second functor is 0-contracted. Hence the first functor is 0-contracted, too.

There is a natural splitting
\[
\bfK(\Nil(\cala,\Phi)) \simeq \bfK(\cala) \vee \widetilde\Nil(\cala,\Phi)
\]
induced by the obvious projection $\Nil(\cala,\Phi)\to\cala$ and its section $A\mapsto (A,0)$. 
Next we show that the  map  induced by the inclusion
\[
\widetilde\Nil(\cala,\Phi)\to \widetilde\Nil(\Idem\cala,\Idem\Phi)
\]
is an equivalence of spectra. 

Denote $\widetilde\Nil_i:=\pi_i\widetilde\Nil$. In the diagram
\[\xymatrix@C-.6em{
0 \ar[r] & K_i(\cala) \ar[r] \ar[d] & K_i (\Nil(\cala,\Phi)) \ar[r]  \ar[d] & \widetilde\Nil_i(\cala,\Phi) \ar[d] \ar[r] & 0
\\
0 \ar[r] & K_i(\Idem \cala) \ar[r]  & K_i (\Nil(\Idem \cala,\Idem \Phi)) \ar[r]   & \widetilde\Nil_i(\Idem \cala,\Idem \Phi) \ar[r] & 0
}
\]
the left and middle vertical arrows are bijections for $i\geq 1$ and injections for $i=0$, by cofinality. 
Since both rows are split exact and the splittings are compatible with the vertical arrows,
also the right arrow is bijective for $i\geq 1$ and injective for $i=0$.

We are left to show surjectivity for $\widetilde\Nil_0$. So let $((A,p),\phi)$
represent an element of
$\widetilde\Nil_0(\Idem\cala,\Idem\Phi)$. Then the same element is also
represented by $((A,p),\phi)\oplus ((A,1-p),0)$ which clearly has a preimage in
$\widetilde\Nil_0(\cala,\Phi)$.

Now the following diagram commutes:
\[
\xymatrix{
\bfK(\Nil(\cala,\Phi)) \ar[d] \ar[r]^\simeq & \bfK(\cala_\Phi[t])\vee \widetilde\Nil(\cala,\Phi) \ar[d]
\\
\bfK(\Nil(\Idem \cala,\Idem \Phi))  \ar[r]^(.35)\simeq & \bfK((\Idem \cala)_{\Idem \Phi}[t])\vee \widetilde\Nil(\Idem \cala,\Idem \Phi)
}
\]
Thinking of all terms as functors in $(\cala,\Phi)$, we know that the lower left term is
0-contracted. It follows from 
Lemma~\ref{lem:c_contracted_and_vee}~\ref{lem:c_contracted_and_vee:contracted} that
$\widetilde\Nil(\cala,\Phi)$ is 0-contracted. Moreover the functor $(\cala,\Phi)\mapsto
\bfK(\cala_\Phi[t])$ is 1-contracted by Theorem~\ref{the:Bass-Heller-Swan_Theorem_for_bfK}
and Lemma~\ref{lem:commuting_[infty]_and_F}. Applying 
again~\ref{lem:c_contracted_and_vee}~\ref{lem:c_contracted_and_vee:contracted} proves the claim.
\end{proof}

Hence we may apply the delooping construction to the functor
\[(\cala,\Phi)\mapsto \bfK(\Nil(\cala,\Phi))\] to obtain a new functor
$\bfKNilinfty(\cala,\Phi)$. It follows from cofinality and Lemma~\ref{lem:alpha_bijective}
that the  map  induced by the inclusion
\[\bfKNilinfty(\cala,\Phi)\to \bfKNilinfty(\Idem\cala,\Idem\Phi)\]
is a weak equivalence.

\begin{proof}[Proof of Theorem~\ref{the:BHS_decomposition_for_non-connective_K-theory}]~%
\ref{the:BHS_decomposition_for_non-connective_K-theory:BHS-iso}
  As $\bfK$ satisfies Condition~\ref{con:condition_bfu}, the same is true for
  $\bfT^t\bfK$ and $\bfN^t_{\pm}\bfE$ and hence for their wedge. So we may apply the
  delooping construction to the transformation $\bfBHS^t$; using compatibility with the
  smash product, we get from 
 Theorem~\ref{the:BHS_decomposition_for_connective_K-theory}~\ref{the:BHS_decomposition_for_connective_K-theory:BHS-iso} 
 a natural homotopy  equivalence
\begin{equation}\label{eq:delooped_BHS_map}
\bfBHS^t[\infty]\colon (\bfT^t\bfK) [\infty] \vee (\bfN^t_+\bfK)[\infty] 
\vee (\bfN^t_-\bfK)[\infty] \xrightarrow{\simeq} (\bfZ^t\bfK)[\infty].
\end{equation}

An application of Lemma~\ref{lem:alpha_bijective} to the functors $z_t[s,s\inv]$ and $z_t[s^{\pm1}]$ implies that 
\begin{equation}\label{eq:canonical_isos}
(\bfZ^t\bfK)[\infty] \cong \bfZ^t \bfKinfty, \quad (\bfN_{\pm}^t\bfK)[\infty]\cong\bfN_\pm^t\bfKinfty.  
\end{equation}
By definition, the mapping torus is a homotopy pushout; the compatibility of the delooping construction with homotopy colimits 
(Theorem~\ref{the:Compatibility_of_the_delooping_construction_with_homotopy_colimits}) implies that the canonical transformation
\[\alpha\colon \bfT^t \bfKinfty \to (\bfT^t \bfK)[\infty]\]
is a weak equivalence. Thus, from \eqref{eq:delooped_BHS_map} we obtain a natural homotopy equivalence 
\[
  \bfa[\infty]\circ \alpha \vee \bfbinfty_+\vee \bfbinfty_- \colon 
  \bfT_{\bfKinfty(\Phi^{-1})} \vee \bfNKinfty(\cala_{\Phi}[t]) \vee \bfNKinfty(\cala_{\Phi}[t^{-1}]) 
  \xrightarrow{\simeq}   \bfKinfty(\cala_{\Phi}[t,t^{-1}]);
  \]

It remains to show that the map $\bfa[\infty]\circ \alpha$ defined in this way agree with the map $\bfainfty$, that is, the map
\[\bfa^t\colon \bfT^t\bfE\to \bfZ^t\bfE\]
for $\bfE=\bfKinfty$ as a functor which satisfies Condition~\ref{con:condition_bfu}. In fact,  the diagram
\[\xymatrix{
\bfT^t (\Omega\bfL\bfE) \ar[rr]^\simeq \ar[d]^{\bfa_{\bfL\bfE}} && \Omega\bfL\bfT^t \bfE \ar[d]^{\bfL\bfa_{\bfE}}\\
\bfZ^t(\Omega\bfL\bfE) \ar[rr]^\cong && \Omega\bfL\bfZ^t\bfE   
}\]
with the canonical horizontal arrows is commutative. Iterating the construction shows that
\[\xymatrix{
\bfT^t (\bfE[\infty]) \ar[rr]^\simeq \ar[d]^{\bfa_{\bfE[\infty]}} && (\bfT^t \bfE)[\infty] \ar[d]^{\bfa_{\bfE}[\infty]}\\
\bfZ^t(\bfE[\infty]) \ar[rr]^\cong && (\bfZ^t\bfE)[\infty]   
}\]
is also commutative. The lower horizontal isomorphisms is \eqref{eq:canonical_isos};
comparing with the proof of
Theorem~\ref{the:Compatibility_of_the_delooping_construction_with_homotopy_colimits} shows
that the upper horizontal map agrees with $\alpha$. This implies the claim.
\\[2mm]~\ref{the:BHS_decomposition_for_non-connective_K-theory:Nil} Here we use that
$K$-theory may be naturally defined on chain complexes: If $\cala$ is an additive
category, we denote by $\Ch(\cala)$ the category of bounded chain complexes over
$\cala$. This category is naturally a Waldhausen category where the weak equivalences are
the chain homotopy equivalences and the cofibrations are the maps which are level-wise
inclusions into a direct summand. As for any Waldhausen category, the $K$-theory spectrum
$\bfK(\Ch(\cala))=:\bfKch(\cala)$ is defined. By the Gillet-Waldhausen theorem, $\bfKch$
is naturally equivalent to $\bfK$ \cite[Proposition~6.1]{Cardenas-Pedersen(1997)}.

Given $(A,f)$ in $\Nil(\cala,\Phi)$, denote by $\chi(A,f)$ the following 1-dimensional chain complex in $\cala_\Phi[t]$:
\[\Phi(A) \xrightarrow{t-f} A\]
This leads to a functor $\chi\colon \Nil(\cala,\Phi)\to \Ch(\cala_\Phi[t])$; 
this functor induces a map
\[\bfK(\chi)\colon \bfK(\Nil(\cala,\Phi)) \to \bfKch(\cala_\Phi[t]).\]
In~\cite[Section~8]{Lueck-Steimle(2013twisted_BHS)} 
it is shown that if $\cala$ is idempotent complete, then $\bfK(\chi)$  is part of a homotopy fiber sequence
\begin{equation}\label{eq:fiber_sequence_with_nil}
 \bfK(\Nil(\cala,\Phi)) \to \bfKch(\cala_\Phi[t^{-1}]) \to \bfKch(\cala_\Phi[t,t\inv]).  
\end{equation}

Now each of the terms in this sequence, as a functor in $(\cala,\Phi)$, satisfies Condition~\ref{con:condition_bfu} 
and the maps in the sequence are compatible transformations. Applying the delooping construction 
to each of the terms leads to a sequence
\begin{equation}\label{eq:fiber_sequence_with_nil_infty}
\bfKNilinfty(\cala,\Phi) \to \bfKchinfty(\cala_\Phi[t]) \to \bfKchinfty(\cala_\Phi[t,t\inv]).  
\end{equation}
Note that by the Gillet-Waldhausen Theorem and Lemma~\ref{lem:commuting_[infty]_and_F}, 
the middle and right terms are naturally homotopy equivalent to $\bfKinfty(\cala_\Phi[t])$ and $\bfKinfty(\cala_\Phi[t,t\inv])$, respectively.

We will show that this sequence is a fibration sequence for any $(\cala,\Phi)$. 
To do this, consider the commutative diagram
\[\xymatrix{
\bfK\circ\Nil \ar[r]\ar[d] & \bfKch\circ \bfZ_+^t \ar[r] \ar[d] & \bfKch\circ \bfZ^t \ar[d]
\\
\bfK\circ\Nil \circ \Idem_t\ar[r]& \bfKch\circ \bfZ_+^t\circ \Idem_t \ar[r]  & \bfKch\circ \bfZ^t\circ \Idem_t
}\]
of functors whose top line at the object $(\cala,\Phi)$ is just \eqref{eq:fiber_sequence_with_nil}, 
and where we define $\Idem_t(\cala,\Phi):=(\Idem(\cala),\Idem(\Phi))$. Notice that the bottom 
line of this diagram, at any object $(\cala,\Phi)$, is a fibration sequence.

We claim that all the functors in this diagram are 1-contracted. In fact, we showed in 
Lemma~\ref{lem:nil_1_contracted} (and its proof) that the left terms are 1-contracted. 
The middle and right upper terms are 1-contracted as the functor $\bfK\simeq \bfKch$ is 1-contracted. 

At the object $(\cala,\Phi)$, the middle vertical arrow is given by the  map  induced by the inclusion
\[\bfKch(\cala_\Phi[t])\to \bfKch((\Idem\cala)_{\Idem(\Phi)}[t]).\]
In particular, by cofinality, it induces an isomorphism in degrees $\geq1$. When precomposed with $Z_+$, it becomes
\begin{equation}\label{eq:transformation_at_Z_plus}
\bfKch((\cala[s])_\Phi[t])\to \bfKch((\Idem(\cala[s]))_{\Idem(\Phi)}[t].
\end{equation}
Now the idempotent completion of $(\cala[s])_\Phi[t]$ is also an idempotent completion of
$(\Idem(\cala[s]))_{\Idem(\Phi)}[t]$. As $\pi_n\bfK$ is invariant under idempotent
completions, we see that \eqref{eq:transformation_at_Z_plus} induces isomorphisms in homotopy
groups of degree $\geq 1$.

The same argument works with $Z_+$ replaced by $Z_-$ and $Z$. We conclude that the
restricted Bass-Heller-Swan maps for the upper and the lower middle terms in the diagram
are isomorphic in degree $\geq1$. Smashing with $(S^1)_+$ preserves connectivity; so the
unrestricted Bass-Heller-Swan maps for the middle terms are also isomorphic in degrees
$\geq2$. Thus, to show that the lower middle term is 1-contracted, we are left to show
split injectivity of the restricted Bass-Heller-Swan map in degree 0.

This map is given by
\begin{multline*}
\pi_0 \bfK((\Idem(\cala)_\Phi[t])\oplus \pi_0 \bfK(\Idem(\cala[s])_\Phi[t]) \oplus \pi_0 \bfK(\Idem(\cala[s\inv])_\Phi[t]) \\
\to \pi_0 \bfK((\Idem(\cala[s,s\inv])_\Phi[t]).
\end{multline*}
Split injectivity holds as $\pi_0\bfK(\cala)\cong\pi_0\bfK(\cala_\Phi[t])$ for any
$(\cala,\Phi)$ (as in the proof of Theorem~\ref{the:Bass-Heller-Swan_Theorem_for_bfK}),
and $\bfKidem$ is 0-contracted.

Thus the lower middle term is 1-contracted and the middle vertical map induces an
isomorphism in degrees $\geq1$. The corresponding statements hold for the right terms in
the diagram, by the very same arguments.

Applying the delooping construction to the whole diagram, we obtain a new diagram whose
top line, at $(\cala,\Phi)$ is given by \eqref{eq:fiber_sequence_with_nil_infty}, and
whose bottom line is still a fiber sequence by 
Theorem~\ref{the:Compatibility_of_the_delooping_construction_with_homotopy_colimits}. 
By Lemma~\ref{lem:alpha_bijective} all the vertical maps are weak homotopy equivalences. We
conclude that the upper line is a fibration sequence. This is what we claimed, for the
upper line, at the object $(\cala,\Phi)$, is just~\eqref{eq:fiber_sequence_with_nil_infty}.

In~\cite[Section~3]{Lueck-Steimle(2013twisted_BHS)}  it is shown that part~\ref{the:BHS_decomposition_for_connective_K-theory:Nil} of 
Theorem~\ref{the:BHS_decomposition_for_connective_K-theory} follows formally 
from~\eqref{eq:fiber_sequence_with_nil}. The same arguments apply to prove that 
part~\ref{the:BHS_decomposition_for_non-connective_K-theory:Nil} of 
Theorem~\ref{the:BHS_decomposition_for_non-connective_K-theory} follows from
\eqref{eq:fiber_sequence_with_nil_infty}.
\end{proof}

\begin{remark}[Schlichting's non-connective $K$-theory spectrum for exact categories]
  \label{rem:Schlichting}
  Notice that $\Nil(\cala,\Phi)$ is an exact category whose exact structure does not come
  from the structure of an additive category.
  Schlichting~\cite{Schlichting(2004deloop_exact)} has defined non-connective $K$-theory
  for exact categories.  It is very likely that Schlichting's non-connected $K$-theory
  applied to the exact category $\Nil(\cala,\Phi)$ is weakly homotopy equivalent to our non-connective
  version $\bfKNilinfty(\cala,\Phi)$ in a natural way. This would follow from 
  Corollary~\ref{cor:turning_a_transformation_into_a_weak_equivalence} 
  if Schlichting's non-connective  $K$-theory of $\Nil(\cala,\Phi)$ is $\infty$-contracted, or, equivalently,
  has a Bass-Heller-Swan decomposition. 
  
  It is conceivable that the twisted Bass-Heller-Swan decomposition for connective $K$-theory, which is described
  in Theorem~\ref{the:BHS_decomposition_for_connective_K-theory} and whose proof is given
  in~\cite{Lueck-Steimle(2013twisted_BHS)}, can be extended directly to the
  non-connective setting described in
  Theorem~\ref{the:BHS_decomposition_for_non-connective_K-theory} using Schlichting's
  non-connective version of $K$-theory for exact categories
and his localization theorem instead of Waldhausen's approximation and fibration theorems.

\end{remark}


\typeout{----------  Section 7: Filtered colimits -----------------}

\section{Filtered colimits}
\label{sec:Filtered_colimits}

Suppose that the small category $\calj$ is filtered, i.e., for any two objects $j_0$ and
$j_1$ there exists a morphism $u \colon j \to j'$ in $\calj$ with the following property:
There exist morphisms from both $j_0$ and $j_1$ to $j$, and for any two morphisms 
$u_0 \colon j_0 \to j$ and $u_1 \colon j_1 \to j$ we have $u \circ j_0 = u \circ j_1$.  Given a
functor $\cala \colon \calj \to \addcat$, its \emph{colimit} $\colim \cala$ in the
category of small categories exists and is in a natural way an additive category.

(We do not need an explicit description, but one can see that
\[\ob(\colim \cala) = \colim (\ob(\cala))\]
and that the abelian group of morphisms from $A$ to $B$ is given by 
\[(\colim\cala)(A,B) = \colim_{j\in \calj}\biggl( \bigoplus_{A_j, B_j} \cala(j)(A_j, B_j)\biggr)\]
where the coproducts range over all objects $A_j, B_j\in\cala(j)$ projecting to $A$ resp.~$B$.)


For a functor $\bfE\colon\addcat\to\Spectra$, we say that \emph{$E$ commutes with filtered
  colimits} if for any $\cala\colon \calj\to\addcat$ the canonical map
\[
\hocolim \bfE \circ \cala \to \bfE(\colim\cala)\]
is a weak homotopy equivalence. By 
Lemma~\ref{lem:diagrams_of_spectra}~\ref{lem:diagrams_of_spectra:pi_and_filtered} 
this is equivalent to saying that
\[
\colim \pi_* \bfE\circ\cala \xrightarrow{\cong} \pi_*\bfE(\colim\cala).
\]

\begin{proposition}\label{prop:compatibility_with_filtered_colimits}
  Suppose that $\bfE$ commutes with filtered colimits and satisfies 
 Condition~\ref{con:condition_bfu}. Then $\bfE[\infty]$ commutes with filtered colimits.
\end{proposition}

Quillen's (or Waldhausen's) connective $K$-theory spectrum $\bfK$ commutes with filtered 
colimits~\cite[(9) on page~20]{Quillen(1973)}. Letting $K^{\infty}_i(\cala) := \pi_i \bigl(\bfKinfty(\cala)\bigr)$, we conclude:

\begin{corollary}\label{cor:K_upper_infty_commutes_with_colimits}
If $\calj$ is filtered and $\cala \colon \calj \to \addcat$ is a covariant functor, then the canonical homomorphism
  \[
  \colim_{\calj} K^{\infty}_i(\cala) \to K^{\infty}_i\bigl(\colim_{\calj} \cala\bigr)
  \]
  is bijective for all $i \in \IZ$.
\end{corollary}

Schlichting proves compatibility of negative $K$-theory with colimits
in~\cite[Corollary~5]{Schlichting(2006)},  cf. Section~\ref{sec:Filtered_colimits}.

\begin{proof}[Proof of Proposition~\ref{prop:compatibility_with_filtered_colimits}]
  Let $\cala\colon \calj\to\addcat$. It follows from the definition of the categories
  $\cala[t^{\pm1}]$ and $\cala[t,t\inv]$ that the the canonical functors
  \[\colim_j (\cala(j)[t^{\pm1}])\to (\colim \cala)[t^{\pm1}], \quad \colim_j
  (\cala(j)[t,t\inv])\to (\colim \cala)[t,t\inv]\] are isomorphisms. It follows that the
  functors $\bfZ_\pm \bfE$ and $\bfZ \bfE$ commute with filtered colimits, too.

Lemma~\ref{lem:diagrams_of_spectra} implies then that for $\bfF\in\{\bfN_\pm, \bfB_r, \bfB, \bfL\}$, $\bfF \bfE$ 
 commutes with filtered colimits. The square
\[\xymatrix{
   {\hocolim \bfE\circ\cala} \ar[rr]^{\hocolim \bfs} \ar[d]^\simeq
   && {\hocolim \Omega\bfL \bfE\circ\cala} \ar[d]^\simeq 
   \\
   \bfE(\colim \cala) \ar[rr]^{\bfs}
   && \Omega \bfL\bfE(\colim\cala)
}\]
with canonical vertical arrows is commutative. Iterating this construction and applying 
Lemma~\ref{lem:diagrams_of_spectra}~\ref{lem:diagrams_of_spectra:commute}, we see that
\[\hocolim \bfE[\infty]\circ\cala \to \bfE[\infty](\colim\cala)\]
is a weak equivalence.
\end{proof}

\begin{remark}
  In~\cite[Theorem~1.8~(i) on page~43]{Bartels-Echterhoff-Lueck(2008colim)} it is shown
  that the Farrell-Jones Isomorphism Conjecture is inherited under filtered colimits of
  groups (with not necessarily injective structure maps), but only for rings as
  coefficients. The same statement remains true if one allows coefficients in additive
  categories, as stated in~\cite[Corollary~0.8]{Bartels-Lueck(2009coeff)}.  The proof
  of~\cite[Theorem~1.8~(i) on page~43]{Bartels-Echterhoff-Lueck(2008colim)} in the $K$-theory case
  carries over   directly as soon as one has Theorem~\ref{cor:K_upper_infty_commutes_with_colimits}
  available, it is needed in the extension of~\cite[Lemma~6.2 on page~61]{Bartels-Lueck(2009coeff)}
  to additive categories. The analog of  Theorem~\ref{cor:K_upper_infty_commutes_with_colimits} for $L$-theory
  is obvious since the $L$-groups of an additive category with involution can be defined elementwise
  instead of referring to the homotopy groups of a spectrum.
\end{remark}


\typeout{----------  Section 8: Homotopy K-theory -----------------}

\section{Homotopy $K$-theory}
\label{sec:Homotopy_K-theory}

Let $\bfE \colon \addcat \to \Spectra$ be a (covariant) functor and let
$F_+ \colon \addcat \to \addcat$ be the functor sending $\cala$ to $\cala[t]$. Denote by 
$\bfF_+\bfE \colon \addcat \to \addcat$ the composite $\bfE \circ F_+$.  The natural inclusion 
$i_+ \colon \cala \to \cala[t]$, which sends a morphism $f \colon A \to B$ to 
$f \cdot t^0 \colon A \to B$, induces a natural transformation $\bfi_+ \colon \bfE \to \bfF_+\bfE$ 
of functors $\addcat \to \Spectra$.  Define inductively the functor 
$\bfF_+^n\colon \addcat \to\Spectra$ by $\bfF^n_+ \bfE := \bfF_+(\bfF^{n-1}_+ \bfE)$ starting with 
$\bfF_+^0 = \bfE$. Define inductively $\bfi^n_+ \colon \bfF^{n-1}\bfE \to \bfF^n\bfE$ by 
$\bfi^n_+ = \bfi_+(\bfi_+^{n-1})$ starting with $\bfi^1_+ := \bfi_+$.  Thus we obtain a sequence of
transformations of functors $\addcat \to \Spectra$
\[
\bfE = \bfF^0_+ \bfE \xrightarrow{\bfi_+^1} \bfF_+^1\bfE \xrightarrow{\bfi_+^2} \bfF_+^2 \bfE
\xrightarrow{\bfi_+^3} \cdots.
\]

\begin{definition}[Homotopy stabilization $\bfH\bfE$]
 \label{def:homotopy_stabilization}
  Define the \emph{homotopy stabilization} 
  \[\bfH\bfE \colon \addcat \to \Spectra\]  
  of  $\bfE \colon \addcat \to \Spectra$ to be the homotopy
  colimit of the sequence above. Let
  \[
  \bfh \colon \bfE \to \bfH\bfE
  \]
  be given by the zero-th structure map of the homotopy colimit.
  We call $\bfE$ \emph{homotopy stable} if $\bfh$ is a weak equivalence.
\end{definition}

This construction has the following basic properties. 
Let $\ev_0^+ \colon \cala_{\Phi}[t] \to \cala$ be the functor sending $\sum_{i \ge 0} f_i \cdot t^i$ to $f_0$.
\begin{lemma}\label{lem:properties_of_bfHbfE}
Let $\bfE \colon \addcat \to \Spectra$ be a covariant functor.

\begin{enumerate}

\item \label{lem:properties_of_bfHbfE:bfHbfE_is_homotopy_stable}
$\bfH\bfE$ is homotopy stable;

\item \label{lem:properties_of_bfHbfE:twisted_homotopy_stable}
Suppose that $\bfE$ is homotopy stable. Let $\cala$ be any additive category with
an automorphism $\Phi \colon \cala \xrightarrow{\cong} \cala$.  Then the maps 
\begin{eqnarray*}
\bfE(\ev_0^+) \colon \bfE(\cala_{\Phi}[t]) & \xrightarrow{\simeq} & \bfE(\cala);
\\
\bfE(\bfi^+) \colon \bfE(\cala)  & \xrightarrow{\simeq} &  \bfE(\cala_{\Phi}[t]),
\end{eqnarray*}
are weak homotopy equivalences.
\end{enumerate}
\end{lemma}
\begin{proof}~\ref{lem:properties_of_bfHbfE:bfHbfE_is_homotopy_stable}
This follows from the definitions since
the obvious map from the homotopy colimit of 
\[
\bfF^0_+ \bfE \xrightarrow{\bfi^1_+} \bfF^1_+\bfE \xrightarrow{\bfi^2_+} \bfF^2_+ \bfE
\xrightarrow{\bfi^3_+} \cdots
\]
to 
\[
\bfF_+^1 \bfE \xrightarrow{\bfi_+^2} \bfF_+^2\bfE \xrightarrow{\bfi_+^3} \bfF^3_+\bfE
\xrightarrow{\bfi^4_+} \cdots
\]
given by applying $\bfi_+$ in each degree is a weak homotopy equivalence.
\\[2mm]~\ref{lem:properties_of_bfHbfE:twisted_homotopy_stable}
Consider an additive category $\cala$ with an automorphism $\Phi \colon \cala \xrightarrow{\cong} \cala$.
Define a functor $j_s \colon \cala_{\Phi}[t] \to \bigl(\cala_{\Phi}[t]\bigr)[s]$ by sending $\sum_{i \ge 0 } f_i \cdot t^i$
to $\sum_{i \ge 0} \bigl(f_i \cdot t^i\bigr) \cdot s^i$. Let
$\ev_{s = 0}$ and $\ev_{s=1}$ respectively be the functors $\bigl(\cala_{\Phi}[t]\bigr)[s] \to  \cala_{\Phi}[t]$
sending a morphism $\sum_{i \ge 0} \bigl(\sum_{j_i \ge 0} f_{j_i,i} \cdot t^{j_i}\bigr) \cdot s^i \colon A \to B$ to
$\sum_{j_0 \ge 0} f_{j_0,0} \cdot t^{j_0} \colon A \to B$ and 
$\sum_{i \ge 0} \sum_{j_i \ge 0} f_{j_i,i} \cdot t^{j_i}  \colon A \to B$  respectively.
Recall that $i_+ \colon \cala \to \cala_{\Phi}[t]$ is the obvious inclusion. 
Then 
\[\ev_{s = 0} \circ j_s=j_+ \circ \ev_0^+, \quad \ev_{s = 1} \circ j_s=\id_{\cala_{\Phi}[t]}.\]

The composite of both $\ev_{s = 0}$ and $\ev_{s = 1}$ with the canonical inclusion 
$k_+  \colon \cala_{\Phi}[t] \to \bigl(\cala_{\Phi}[t]\bigr)[s]$ is the identity. Since $\bfE$
is homotopy stable by assumption, the map 
$\bfE(k_+) \colon \bfE\bigl(\cala_{\Phi}[t]\bigr) \to \bfE\bigl((\cala_{\Phi}[t])[s]\bigr)$ 
is a weak equivalence. Hence the functors $\ev_{s = 0}$ and $\ev_{s = 1}$
induce same homomorphism after applying $\pi_i \circ \bfE$ for $i \in \IZ$.  This implies
that the composite
\[\pi_i\bigl(\bfE(\cala_{\phi}[t])\bigr) \xrightarrow{\pi_i(\bfE(\ev_0^+))} \pi_i\bigl(\bfE(\cala)\bigr)  
\xrightarrow{\pi_i(\bfE(i_+))} \pi_i\bigl(\bfE(\cala_{\phi}[t])\bigr)\]
is the identity. Since $\ev_0^+ \circ i_+ = \id_{\cala}$, also the composite
\[ \pi_i\bigl(\bfE(\cala)\bigr)  \xrightarrow{\pi_i(\bfE(i_+))}
\pi_i\bigl(\bfE(\cala_{\phi}[t])\bigr)  \xrightarrow{\pi_i(\bfE(\ev_0^+))} \pi_i\bigl(\bfE(\cala)\bigr)
\]
is the identity. Hence $\bfE(\ev_0^+) \colon \bfE(\cala_{\Phi}[t]) \to \bfE(\cala)$ is a weak homotopy equivalence.
\end{proof}

\begin{remark}[Universal property of $\bfH$]
  \label{rem:universal_property_of_ngH}
  Notice that
  Lemma~\ref{lem:properties_of_bfHbfE}~\ref{lem:properties_of_bfHbfE:bfHbfE_is_homotopy_stable}
  says that up to weak homotopy equivalence the transformation $\bfh \colon \bfE\to \bfH\bfE$
  is universal (from the left) among transformations $\bff \colon \bfE \to \bfF$ to  homotopy stable
  functors $\bfF \colon \addcat \to \Spectra$ since we obtain a commutative square whose lower
  vertical arrow is a weak homotopy equivalence
  \[
  \xymatrix{\bfE \ar[r]^{\bfh} \ar[d]^{\bff} & \bfH\bfE \ar[d]^{\bfH\bff}
    \\
    \bfF \ar[r]_{\bfh}^{\simeq} & \bfH\bfF
  }
  \]
\end{remark}

Lemma~\ref{lem:properties_of_bfHbfE}~\ref{lem:properties_of_bfHbfE:twisted_homotopy_stable} 
essentially says that homotopy stable automatically implies homotopy stable in the twisted case.

\begin{definition}[Homotopy $K$-theory]
\label{homotopy_K-theory}
Define the homotopy $K$-theory functors
\[
\bfH\bfK, \bfH\bfKinfty \colon \addcat \to \Spectra
\]
to be the homotopy stabilization in the sense of 
Definition~\ref{def:homotopy_stabilization} of the functors $\bfK, \bfKinfty \colon \addcat \to \Spectra$.
\end{definition}

\begin{lemma} \label{lem:Compatibility_of_bfH_and[infty]}
The functor $\bfH\bfK$ is $1$-contracted and there is a weak equivalence
\[
\bfH\bfK[\infty] \xrightarrow{\simeq} \bfH\bfKinfty.
\]
\end{lemma}
\begin{proof}
As 
\[\pi_* \bfH\bfE(\cala) \cong \colim_n \pi_* \bfE(\cala[t_1, \dots, t_n])\]
we conclude that $\bfH\bfE$ is $c$-contracted provided that $E$ is $c$-contracted. 
Applying Theorem~\ref{the:Bass-Heller-Swan_Theorem_for_bfK} we see that $\bfH\bfK$ is 1-contracted. Also $\bfH\bfKinfty$ is $\infty$-contracted. 
The claim now follows from Lemma~\ref{lem:alpha_bijective}.
\end{proof}

\begin{lemma} [Bass-Heller-Swan for homotopy $K$-theory]
\label{lem_Bass-Heller-Swan_for_homotopy_K-theory}
Let $\cala$ be an additive category with an automorphism $\Phi \colon \cala \xrightarrow{\cong} \cala$.
Then we get for all $n \in \IZ$ a weak homotopy equivalence
\[
  \bfainfty  \colon 
  \bfT_{\bfKinfty(\Phi^{-1})} 
  \xrightarrow{\simeq}   \bfH\bfKinfty(\cala_{\Phi}[t,t^{-1}]).
\]
\end{lemma}
\begin{proof}
We conclude from
  Theorem~\ref{the:BHS_decomposition_for_non-connective_K-theory}~%
\ref{the:BHS_decomposition_for_non-connective_K-theory:BHS-iso}
  and the fact that the Bass-Heller-Swan map is compatible with homotopy colimits in the
  spectrum variable and $\bfH\bfKinfty$ is defined as a homotopy colimit in terms of
  $\bfKinfty$ that there is a weak equivalence of spectra, natural in $(\cala,\Phi)$,
  \[
  \bfa \vee \bfb_+\vee \bfb_- \colon \bfT_{\bfH\bfKinfty(\Phi^{-1})} \vee
  \bfN\bfH\bfKinfty(\cala_{\Phi}[t]) \vee \bfN\bfH\bfKinfty(\cala_{\Phi}[t^{-1}]) \xrightarrow{\simeq}
  \bfH\bfKinfty(\cala_{\Phi}[t,t^{-1}]).
  \]
Since all the homotopy groups of the terms
$\bfN\bfH\bfKinfty(\cala_{\Phi}[t])$ and $\bfN\bfH\bfKinfty(\cala_{\Phi}[t^{-1}])$ vanish by
Lemma~\ref{lem:properties_of_bfHbfE}~\ref{lem:properties_of_bfHbfE:twisted_homotopy_stable}, 
Lemma~\ref{lem_Bass-Heller-Swan_for_homotopy_K-theory} follows.
\end{proof}

\begin{remark}[Identification with Weibel's definition]
  \label{rem:Identification_with_Weibel's_definition}
Weibel has defined a  version of homotopy $K$-theory for a ring $R$ by a simplicial construction
  in~\cite{Weibel(1989)}.   It is not hard to
  check using Remark~\ref{rem:universal_property_of_ngH}, which applies also to the
  constructions of~\cite{Weibel(1989)} instead of
  $\bfH$, that $\pi_i(\bfH\bfKidem(\calr))$ and $\pi_i(\bfH\bfKinfty(\calr))$ can be
  identified with the one in~\cite{Weibel(1989)}, if
  $\calr$ is a skeleton of the category of finitely generated free $R$-modules.
\end{remark}


\typeout{----------  Section 8: The Farrell-Jones Conjecture for homotopy K-theory  -----------------}

\section{The Farrell-Jones Conjecture for homotopy  $K$-theory}
\label{sec:The_Farrell-Jones_Conjecture_for_homotopy_K-theory}

The Farrell-Jones Conjecture for (non-connective) homotopy $K$-theory has been treated for rings
in~\cite{Bartels-Lueck(2006)}. Meanwhile it has turned out to be useful to formulate the
Farrell-Jones Conjecture for additive categories as coefficients since then one has much
better inheritance properties, see for instance~\cite{Bartels-Lueck(2009coeff)}
and~\cite{Bartels-Reich(2007coeff)}. The Farrell-Jones Conjecture for (non-connective)
$K$-theory for additive categories is true for a group $G$, if for any additive
$G$-category $\cala$ the assembly map
\[
H_n^G(\EGF{G}{\VCyc};\bfKinfty_{\cala}) \to H_n^G(\pt;\bfKinfty_{\cala}) = K_n(\int_G \cala)
\]
is bijective for all $n \in \IZ$, where $\EGF{G}{\VCyc}$ is 
the classifying space for the family of virtually cyclic groups.
If one replaces $\bfKinfty \colon \addcat \to \Spectra$
by the functor $\bfH\bfKinfty \colon \Spectra \to \addcat$ and $\EGF{G}{\VCyc}$
by the classifying space $\EGF{G}{\Fin}$ for the family of finite subgroups,
one obtains the Farrell-Jones Conjecture for
algebraic (non-connective) homotopy $K$-theory with coefficients in additive categories.
It predicts the bijectivity of 
\[
H_n^G(\EGF{G}{\Fin};\bfH\bfKinfty_{\cala}) \to H_n^G(\pt;\bfH\bfKinfty_{\cala})  = \KH_n(\int_G \cala)
\]
for all $n \in \IZ$, where $\KH_n(\calb)$ denotes $\pi_n\bigl((\bfH\bfKinfty(\calb)\bigr)$
for an additive category $\calb$.

For the following result, denote by $\FJH(G)$ the 
statement ``The Farrell-Jones Conjecture for algebraic homotopy $K$-theory with coefficients in additive categories holds for $G$''.

\begin{theorem}[Farrell-Jones Conjecture for homotopy $K$-theory]
\label{the:Farrell-Jones_Conjecture_for_homotopy_K-theory}
\
\begin{enumerate}

\item \label{the:Farrell-Jones_Conjecture_for_homotopy_K-theory:extensions}
Let $1 \to K \to G \to Q \to 1$ be an extensions of groups.
If $\FJH(Q)$ and $\FJH(p\inv(H))$ for any finite subgroup $H \subseteq Q$, then $\FJH(G)$;

\item \label{the:Farrell-Jones_Conjecture_for_homotopy_K-theory:trees} If $G$
  acts on a tree $T$ such that $\FJH(G_x)$ for every stabilizer group $G_x$ of $x \in T$, then $\FJH(G)$;

\item \label{the:Farrell-Jones_Conjecture_for_homotopy_K-theory:From_K_to_HK} 
  If $G$ satisfies the Farrell-Jones Conjecture for algebraic $K$-theory with coefficients in
  additive categories, then $\FJH(G)$.
\end{enumerate}
\end{theorem}
\begin{proof}~\ref{the:Farrell-Jones_Conjecture_for_homotopy_K-theory:extensions} This follows
from~\cite[Corollary~4.4]{Bartels-Lueck(2006)}.
\\[1mm]~\ref{the:Farrell-Jones_Conjecture_for_homotopy_K-theory:trees} It is easy to check
that the arguments in Bartels-L\"uck~\cite{Bartels-Lueck(2006)} carry over from rings to
additive categories since they are on the level of equivariant homology theories.
\\[1mm]~\ref{the:Farrell-Jones_Conjecture_for_homotopy_K-theory:From_K_to_HK} 
It follows from~\cite[Remark~1.6]{Davis-Quinn-Reich(2011)} that the assembly map
\[
H_n^G(\EGF{G}{\VCyc_I};\bfKinfty_{\cala})\to  H_n^G(\pt;\bfKinfty_{\cala}) = K_n(\int_G \cala)
\]
is bijective for $n \in \IZ$, where we have replaced $\VCyc$ by the smaller family of
subgroups $\VCyc_I$ of virtually cyclic subgroups of type $I$, i.e., of subgroups which
are either finite or admit an epimorphism to $\IZ$ with finite kernel. Since
$\bfH\bfKinfty$ is given by a specific homotopy colimit, the assembly map is required to
be bijective for all additive $G$-categories $\cala$ and is compatible with homotopy
colimits in the spectrum variable, we conclude that
\[
H_n^G(\EGF{G}{\VCyc_I};\bfH\bfKinfty_{\cala}) \to H_n^G(\pt;\bfH\bfKinfty_{\cala})  = \KH_n(\int_G \cala)
\]
is bijective for $n \in \IZ$. In order to replace $\VCyc_I$ by $\Fin$, we have to show
in view of the Transitivity Principle, see for 
instance~\cite[Theorem~1.11]{Bartels-Farrell-Lueck(2011cocomlat)}
or~\cite[Theorem~1.5]{Bartels-Lueck(2007ind)},
that for any virtually cyclic group $V$ of type $I$ the assembly map
\[
H_n^G(\EGF{V}{\Fin};\bfH\bfKinfty_{\cala}) \to H_n^G(\pt;\bfH\bfKinfty_{\cala}) = \KH_n(\int_V \cala)
\]
is bijective for all $n \in \IZ$.  This follows for $V = \IZ$ from
Lemma~\ref{lem_Bass-Heller-Swan_for_homotopy_K-theory} since the assembly map appearing
above can be identified with the map appearing in
Lemma~\ref{lem_Bass-Heller-Swan_for_homotopy_K-theory}. The general case of an extension
$1 \to F \to V \to \IZ$ can be reduced to  case $V = \IZ$ by 
assertion~\ref{the:Farrell-Jones_Conjecture_for_homotopy_K-theory:extensions}.
\end{proof}

\begin{remark}[Wreath products]
  \label{rem:wreath_products}
  We can also consider the versions ``with finite wreath products'', i.e., we require for
  a group $G$ that the Farrell-Jones Conjecture is not only satisfied for $G$ itself, but
  for all wreath products of $G$ with finite groups,   see for  instance~\cite{Farrell-Roushon(2000)}.  
  The advantage of this version is that it is
  inherited to overgroups of finite
  index. This follows from the fact that an overgroup
  $H$ of finite index of $G$ can be embedded into a wreath product $G \wr F$ 
  for a finite group $F$, see~\cite[Section~2.6]{Dixon-Mortimer(1996)}.
  Theorem~\ref{the:Farrell-Jones_Conjecture_for_homotopy_K-theory} remains true for
  the version with finite wreath products, where
  assertion~\ref{the:Farrell-Jones_Conjecture_for_homotopy_K-theory:extensions} can be
  reduced to the statement that for an extension $1 \to G \to Q \to 1$ the Farrell-Jones
  Conjecture with wreath products holds for $G$ if it holds for $K$ and $Q$.
\end{remark}

\begin{remark}[Status of the Farrell-Jones Conjecture for homotopy $K$-theory]
  Because of assertion~\ref{the:Farrell-Jones_Conjecture_for_homotopy_K-theory:extensions} 
  and~\ref{the:Farrell-Jones_Conjecture_for_homotopy_K-theory:trees} of
  Theorem~\ref{the:Farrell-Jones_Conjecture_for_homotopy_K-theory}, the class of groups
  for which the Farrell-Jones Conjecture for homotopy algebraic $K$-theory is known is
  larger than the class for the Farrell-Jones Conjecture for algebraic $K$-theory. For
  instance, elementary amenable groups satisfy the version for homotopy $K$-theory, just
  adapt the argument in~\cite[Theorem~1.3~(i),
  Lemma~2.12]{Bartels-Lueck-Reich(2008appl)}. On the other hand, it is not known whether
  the Farrell-Jones Conjecture for algebraic $K$-theory holds for the semi-direct product
  $\IZ[1/2] \rtimes \IZ$, where $\IZ$ acts by multiplication with $2$ on $\IZ[1/2] \rtimes
  \IZ$. 

  To summarize, the Farrell-Jones Conjecture for homotopy $K$-theory for coefficients in
  additive categories with finite wreath products has the following properties:
  \begin{itemize}
  \item It is known for elementary amenable groups, hyperbolic, CAT(0)-groups, $GL_n(R)$
    for a ring $R$ whose underlying abelian group is finitely generated, arithmetic groups
    over number fields, arithmetic groups over global fields, cocompact lattices in almost connected Lie groups,
    fundamental groups of  (not necessarily compact) $3$-manifolds (possibly with boundary), and one-relator groups;
  \item It is closed under taking subgroups;
  \item It is closed under taking finite direct products and finite free products;
  \item It is closed under directed colimits (with not necessarily injective) structure
    maps;
  \item It is closed under extensions as explained in Remark~\ref{rem:wreath_products};
  \item It has the tree property, see
    Theorem~\ref{the:Farrell-Jones_Conjecture_for_homotopy_K-theory}%
~\ref{the:Farrell-Jones_Conjecture_for_homotopy_K-theory:trees};
  \item It is closed under passing to overgroups of finite index.
  \end{itemize}
  This follows from the results above
  and~\cite{Bartels-Farrell-Lueck(2011cocomlat)},~\cite{Bartels-Lueck(2012annals)},~\cite{Bartels-Lueck-Reich(2008hyper)},%
~\cite{Bartels-Lueck-Reich-Rueping(2012)}, and~\cite{Rueping(2013)}.
\end{remark}

 \begin{remark}[Implications of the homotopy $K$-theory version to the $K$-theory version]
 \label{rem:Implications_of_the_homotopy_K-theory_version_to_the_K-theory_version}
   We have already seen above that the Farrell-Jones Conjecture for $K$-theory
 with coefficients in additive categories implies the Farrell-Jones Conjecture for homotopy $K$-theory
 with coefficients in additive categories. Next we discuss some cases,  where 
 the Farrell-Jones Conjecture for homotopy $K$-theory with coefficients in the  ring $R$
 gives implications for the injectivity part of the Farrell-Jones Conjecture for $K$-theory
 with coefficients in the ring $R$. These all follow by inspecting for a ring $R$ the following commutative diagram
 \[
\xymatrix{H_n^G(\EGF{G}{\VCyc};\bfKinfty_R) \ar[r] 
&
H_n^G(\pt;\bfKinfty_{R}) = K_n(RG) \ar[dd]^{\KH(\bfh^{\infty})}
\\
H_n^G(\EGF{G}{\Fin};\bfKinfty_R) \ar[u]^f \ar[d]_{h^{\infty}}
&
\\
H_n^G(\EGF{G}{\Fin};\bfH\bfKinfty_R) \ar[r]
&
H_n^G(\pt;\bfH\bfKinfty_{R}) = \KH_n(RG)
}
\]
where the two vertical arrows pointing downwords are induced by the transformation
$\bfh^{\infty} \colon \bfKinfty \to \bfH\bfKinfty$, the map $f$ is induced by the
inclusion of families $\Fin \subseteq \VCyc$ and the two horizontal arrows are the
assembly maps for $K$-theory and homotopy $K$-theory.

Suppose that $R$ is regular and the order of any finite subgroup of $G$ is invertible in $R$.
Then the two left vertical arrows are known to be bijections. This follows
for $f$ from~\cite[Proposition~2.6 on page~686]{Lueck-Reich(2005)}
and for $h^{\infty}$  from~\cite[Lemma~4.6]{Davis-Lueck(1998)} and  the fact that $RH$ is regular for all finite subgroups
$H$ of $G$ and hence $K_n(RH) \to \KH_n(RH)$ is bijective for all $n \in \IZ$.
Hence the (split) injectivity  
of the lower horizontal arrow implies the  (split) injectivity  
of the upper horizontal  arrow.

Suppose that $R$ is regular. Then the two left vertical arrows are rational bijections, This follows
for $f$ from~\cite[Theorem~0.3]{Lueck-Steimle(2013splitasmb)}.
To show it for $h^{\infty}$ it suffices 
because of~\cite[Lemma~4.6]{Davis-Lueck(1998)} to
show that $K_n(RH) \to \KH_n(RH)$ is rationally bijective for each finite group $H$ and $n
\in \IZ$.  By the version of the spectral sequence appearing in~\cite[1.3]{Weibel(1989)}
for non-connective $K$-theory, it remains to show that $N^pK_n(RH)$ vanishes rationally
for all $n \in \IZ$. Since $R[t]$ is regular if $R$ is, this boils down to show that
$\NK_p(RH)$ is rationally trivial for any regular ring $R$ and any finite group $H$.  
This reduction can also be shown by proving directly that the structure maps
of the system of spectra appearing  in the Definition~\ref{def:homotopy_stabilization}
 of $\bfH\bfK$ are weak homotopy equivalences.
The proof that $ \NK_p(RH)$ is rationally trivial for any regular ring $R$ and any finite group $H$
can be found for instance  in~\cite[Theorem~9.4]{Lueck-Steimle(2013splitasmb)}. 
Hence the upper horizontal arrow is rationally injective if the lower horizontal arrow is rationally injective.
\end{remark}


\typeout{-----------------------  References ------------------------}




\end{document}